\documentclass{amsart}
\usepackage{graphicx,accents}
\newcommand\R{\mathbb R}
\newcommand\x{\times}
\newcommand\pd[2]{\frac{\partial #1}{\partial #2}}
\DeclareMathOperator{\grad}{grad}
\DeclareMathOperator{\curl}{curl}
\DeclareMathOperator{\dist}{dist}
\DeclareMathOperator{\diam}{diam}
\let\div\undefined
\DeclareMathOperator{\div}{div}
\newcommand{\exd}{d}
\newcommand\T{\mathcal T}
\newcommand{\<}{\langle}
\renewcommand{\>}{\rangle}
\newtheorem{thm}{Theorem}[section]
\newtheorem{defn}[thm]{Definition}
\newtheorem{lem}[thm]{Lemma}
\theoremstyle{remark}
\newtheorem{remark}[thm]{Remark}
\numberwithin{equation}{section}

\begin{document}

\title{On the consistency\\of the combinatorial codifferential}
\author[D. N. Arnold]{Douglas N. Arnold}
\address{School of Mathematics, University of Minnesota, Minneapolis, Minnesota 55455}
\email{arnold@umn.edu}
\thanks{The work of the first author was supported by NSF grant DMS-1115291.}

\author[R. S. Falk]{Richard S. Falk}
\address{Department of Mathematics, Rutgers University, Piscataway, New Jersey 08854}
\email{falk@math.rutgers.edu}
\thanks{The work of the second author was supported by NSF grant DMS-0910540.}

\author[J. Guzm\'an]{Johnny Guzm\'an}
\address{Division of Applied Mathematics, Brown University, Providence, Rhode Island 02912}
\email{johnny\_guzman@brown.edu}

\author[G. Tsogtgerel]{Gantumur Tsogtgerel}
\address{Department of Mathematics and Statistics, McGill University, Montreal, Quebec, Canada  H3A 0B9}
\email{gantumur@math.mcgill.ca}
\thanks{The work of the fourth author was supported by an NSERC Discovery Grant and an FQRNT Nouveaux Chercheurs Grant.}



\date{December 18, 2012}

\begin{abstract}
In 1976, Dodziuk and Patodi employed
Whitney forms to define a combinatorial codifferential operator
on cochains, and  they raised the question whether
it is consistent
in the sense that for a smooth enough differential form
the combinatorial codifferential of the associated cochain converges
to the exterior codifferential of the form as the triangulation is refined.  In 1991,
Smits proved this to be the case for the combinatorial codifferential applied
to $1$-forms in two dimensions under the additional assumption that the initial triangulation
is refined in a completely regular fashion, by dividing each triangle into
four similar triangles.  In this paper we extend the
result of Smits to arbitrary dimensions, showing that the combinatorial
codifferential on $1$-forms is consistent if the triangulations are uniform or
piecewise uniform in a certain precise sense.  We also show that this restriction on the
triangulations is needed, giving a counterexample in which a
different regular refinement procedure, namely
Whitney's standard subdivision, is used.  Further, we show by numerical example
that for $2$-forms in three dimensions, the combinatorial codifferential is not
consistent, even for the most regular subdivision process.
\end{abstract}

\maketitle

\section{Introduction}\label{s:introduction}

Let $M$ be an $n$-dimensional polytope in $\R^n$, triangulated by a simplicial complex $\T_h$
of maximal simplex diameter $h$,
which we orient by fixing an order for the vertices.
(Although we restrict ourselves to polytopes for simplicity,
several of the results below can easily be extended to triangulated Riemannian manifolds.)  We
denote by $\Lambda^k=\Lambda^k(M)$ the space of smooth differential $k$-forms
on $M$.  The Euclidean inner product restricted to $M$
determines the Hodge star operator $\Lambda^k\to \Lambda^{n-k}$, and the inner product
on $\Lambda^k$ given by $\< u,v\>=\int u\wedge \star v$.  The space
$L^2\Lambda^k$ is the completion of $\Lambda^k$ with respect to this norm,
i.e., the space of differential $k$-forms with coefficients in $L^2$.
We then define $H\Lambda^k$ to be the space of forms $u$ in $L^2\Lambda^k$ whose exterior derivative
$\exd u$, which may be understood in the sense of distributions,
belongs to $L^2\Lambda^{k+1}$.  These spaces combine to form the $L^2$ de~Rham complex
$$
0\to H\Lambda^0\xrightarrow{\exd} H\Lambda^1\xrightarrow{\exd}\cdots\xrightarrow{\exd} H\Lambda^n\to 0.
$$
Viewing the exterior derivative $\exd$ as an unbounded operator $L^2\Lambda^k$
to $L^2\Lambda^{k+1}$ with domain $H\Lambda^k$, we may define its adjoint
$\exd^*$.  Thus a differential $k$-form $u$ belongs
to the domain of $\exd^*$ if the operator $v\mapsto \<u,\exd v\>_{L^2\Lambda^k}$ is
bounded on $L^2\Lambda^{k-1}$, and then
$$
\<\exd^* u,v\>_{L^2\Lambda^{k-1}} = \<u,\exd v\>_{L^2\Lambda^k}, \quad v\in H\Lambda^{k-1}.
$$
In particular, every $u$ which is smooth and supported in the interior of $M$ belongs
to the domain of $\exd^*$ and $\exd^* u= (-1)^{k(n-k+1)}\star \exd\star u$.

Let $\Delta_k(\T_h)$ denote the set of $k$-dimensional simplices of $\T_h$.
We denote by $C_k(\T_h)$ the space of formal linear combinations of elements
of $\Delta_k(\T_h)$ with real coefficients, the space of $k$-chains, and
by $C^k(\T_h)=C_k(\T_h)^*$ the space of $k$-cochains.
The coboundary maps $d^c:C^k(\T_h)\to C^{k+1}(\T_h)$ then determine the cochain complex.
The \emph{de~Rham map} $R_h$ maps $\Lambda^k$
onto $C^k(\T_h)$ taking a differential $k$-form $u$ to the cochain
\begin{equation}\label{dRm}
R_hu: C_k(\T_h)\to \R, \quad c\mapsto \int_c u.
\end{equation}
The canonical basis for $C^k(\T_h)$
consists of the cochains $a_\tau$, $\tau\in\Delta_k(\T_h)$, where $a_\tau$ takes
the value $1$ on $\tau$ and zero on the other elements of $\Delta_k(\T_h)$.
The associated \emph{Whitney form} is given by
$$
W_h a_\tau = k!\sum_{i=0}^k (-1)^i
  \lambda_i\, d\lambda_0\wedge\cdots\wedge\widehat{d\lambda_i}\wedge\cdots\wedge d\lambda_k,
$$
where $\lambda_0,\ldots,\lambda_k$ are the piecewise linear basis functions associated
to the vertices of the simplex listed, i.e., $\lambda_i$ is the continuous piecewise linear function
equal to $1$ at the $i$th vertex of $\tau$ and vanishing at all the other vertices of the
triangulation.  The span of $W_ha_\tau$, $\tau\in\Delta_k(\T_h)$, defines the space of $\Lambda^k_h$
of Whitney $k$-forms.  Its elements are piecewise affine differential $k$-forms
which belong to $H\Lambda^k$ and
satisfy $\exd \Lambda^k_h\subset \Lambda^{k+1}_h$.  Thus the Whitney forms comprise
a finite-dimensional subcomplex of the $L^2$ de~Rham complex called the Whitney complex:
$$
0\to \Lambda^0_h\xrightarrow{\exd} \Lambda^1_h\xrightarrow{\exd}\cdots\xrightarrow{\exd} \Lambda^n_h\to 0.
$$
The \emph{Whitney map} $W_h$ maps $C^k(\T_h)$ isomorphically onto $\Lambda^k_h$ and satisfies
\begin{equation}\label{ccm}
 W_hd^c c = \exd W_h c, \quad c\in C^k(\T_h), 
\end{equation}
i.e., is a cochain isomorphism of
the cochain complex onto the Whitney complex.  Although Whitney $k$-forms need not be continuous,
each has a well-defined trace on the simplices in $\Delta_k(\T_h)$, so the de~Rham map
\eqref{dRm} is defined for $u\in\Lambda^k_h$.  The Whitney map is a one-sided inverse
of the de~Rham map: $R_hW_h c = c$ for $c\in C^k(\T_h)$.  The reverse
composition $\pi_h=W_hR_h:\Lambda^k\to \Lambda^k_h$
defines the \emph{canonical projection} into $\Lambda^k_h$.

In \cite{dodziuk} and \cite{dodziuk-patodi},
Dodziuk and Patodi defined an inner product on cochains by declaring the Whitney
map to be an isometry:
\begin{equation}\label{isom}
\< a,b\> = \<W_h a, W_h b\>_{L^2\Lambda^k}, \quad a,b\in C^k(\T_h).
\end{equation}
They then used this inner product to define the adjoint $\delta^c$ of the coboundary:
\begin{equation}\label{cochainadj}
\<\delta^c a,b\> = \<a, d^c b\>, \quad  a,b\in C^k(\T_h).
\end{equation}
Since the coboundary operator $d^c$ may be viewed as a combinatorial version of
the differential operator of the de~Rham complex, its adjoint $\delta^c$ may be
viewed as a combinatorial codifferential, and together they define
the combinatorial Laplacian on cochains given by
$$
\Delta^c= d^c\delta^c+\delta^c d^c:C^k(\T_h)\to C^k(\T_h).
$$
The work of Dodziuk and Patodi
concerned the relation between the eigenvalues of this combinatorial
Laplacian and those of the Hodge Laplacian.

Dodziuk and Patodi asked whether the combinatorial codifferential $\delta^c$ is
a consistent approximation of $\exd^*$ in the sense that if we have a sequence
of triangulations $\T_h$ with maximum simplex diameter tending to zero and
satisfying some regularity restrictions, then
\begin{equation}\label{cochainconsist}
\lim_h \|W_h \delta^c R_hu - \exd^* u\|=0,
\end{equation}
for sufficiently smooth $u\in \Lambda^k$ belonging to the domain of $\exd^*$.  Here and henceforth
the norm $\|\,\cdot\,\|$ denotes the $L^2$ norm.

Since $C^k(\T_h)$ and $\Lambda^k_h$ are isometric, we may state this question
in terms of Whitney forms, without invoking cochains.
Define the Whitney codifferential $\exd^*_h:\Lambda^k_h\to \Lambda^{k-1}_h$ by
\begin{equation}\label{adj}
\<\exd^*_h u, v\>_{L^2\Lambda^{k-1}} = \<u,\exd v\>_{L^2\Lambda^k}, \quad u \in \Lambda^k_h,\ v\in \Lambda^{k-1}_h.
\end{equation}
Combining \eqref{ccm}, \eqref{isom}, and \eqref{cochainadj}, we see that $\exd^*_h = W_h\delta^c W_h^{-1}$.
Therefore, $W_h\delta^c R_h\\= d^*_h\pi_h$, and the question of consistency becomes whether
\begin{equation}\label{consist}
\lim_h \|\exd^*_h\pi_hu - \exd^* u\|=0,
\end{equation}
for smooth $u$ in the domain of $\exd^*$.

In Appendix~II of \cite{dodziuk-patodi}, the authors suggest a counterexample to \eqref{consist}
for $1$-forms (i.e., $k=1$) on a two-dimensional manifold, but, as pointed out by Smits \cite{smits},
the example is not valid, and the question has remained open.  Smits himself considered the question,
remaining in the specific case of $1$-forms on a two-dimensional manifold, and restricting
himself to a sequence of triangulations obtained by \emph{regular standard subdivision}, meaning that
the triangulation is refined by dividing each triangle into four similar triangles by connecting
the midpoints of the edges, resulting in a piecewise uniform sequence of triangluations.
See Figure~\ref{f:puniform} for an example.
In this case, Smits proved that \eqref{cochainconsist} or, equivalently,
\eqref{consist} holds.

Smits's result leaves open various questions.
Does the consistency of the $1$-form codifferential on regular meshes in two dimensions extend to
\begin{itemize}
 \item Mesh sequences which are not obtained by regular standard subdivision?
 \item More than two dimensions?
 \item The combinatorial codifferential on $k$-forms with $k>1$?
\end{itemize}
In this paper we show that the answer to the second question is affirmative, but the answers
to the first and third are negative.  More precisely, in Section~\ref{s:counterexample}
we present a simple counterexample to consistency for a quadratic $1$-form
on the sequence of triangulations shown in 
Figure~\ref{f:crisscross}.  While these meshes are not obtained by regular
standard subdivision, they may be obtained by another systematic subdivision process,
\emph{standard subdivision}, as defined by Whitney in \cite[Appendix II, \S~4]{whitney}.
Next, in Section~\ref{s:superconv}, we recall a definition of \emph{uniform} triangulations in
$n$-dimensions which was formulated in the study of superconvergence of finite element
methods, and we use the superconvergence theory to extend Smits's result on
the consistency of the combinatorial codifferential on $1$-forms to $n$-dimensions, for
triangulations that are uniform or piecewise uniform.  In Section~\ref{s:experiments},
we provide computational confirmation of these results, both positive and negative.
Finally, in Section~\ref{s:2-forms}, we numerically explore
the case of $2$-forms in three dimensions and find that the combinatorial
codifferential is inconsistent, even
for completely uniform mesh sequences.

\section{A counterexample to consistency}\label{s:counterexample}
We take as our domain $M$ the square $(-1,1)\x (-1,1)\subset \R^2$, and
as initial triangulation the division into four triangles obtained
via drawing the two diagonals.  We refine a triangulation by
subdividing each triangle into four using standard subdivision.  In this
way we obtain the sequence of \emph{crisscross triangulations} shown in Figure~\ref{f:crisscross},
with the $m$th triangulation consisting of $4^m$ isoceles right triangles.
We index the triangulation by the diameter of its elements, so we denote
the $m$th triangulation by $\T_h$ where $h=4/2^m$. Using this triangulation, the authors of \cite{DMR91} showed that superconvergence does not hold for piecewise linear Lagrange elements.  
\begin{figure}[htb]
\centerline{%
 \begin{tabular}{cc}
  \includegraphics[width=1.2in]{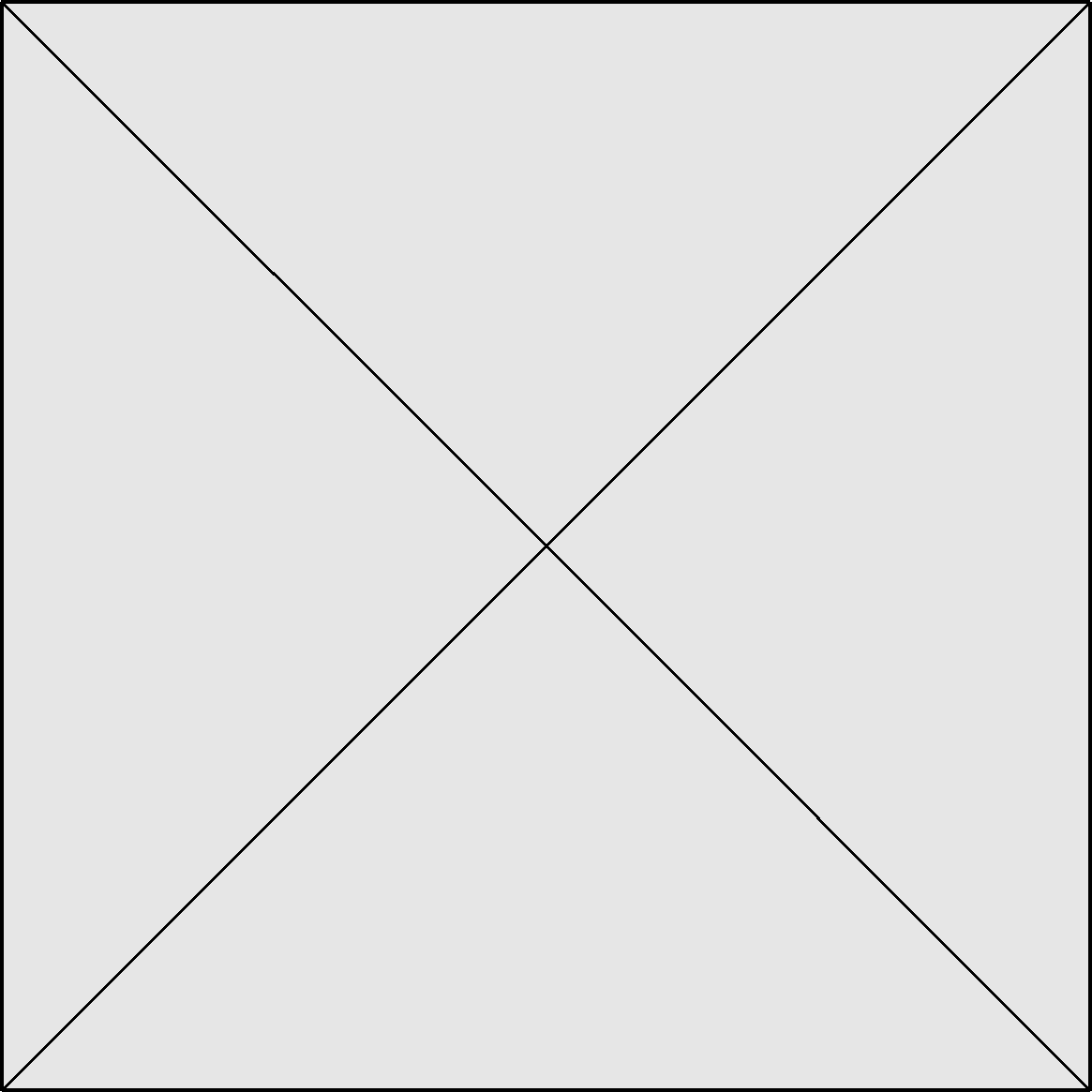} &
  \includegraphics[width=1.2in]{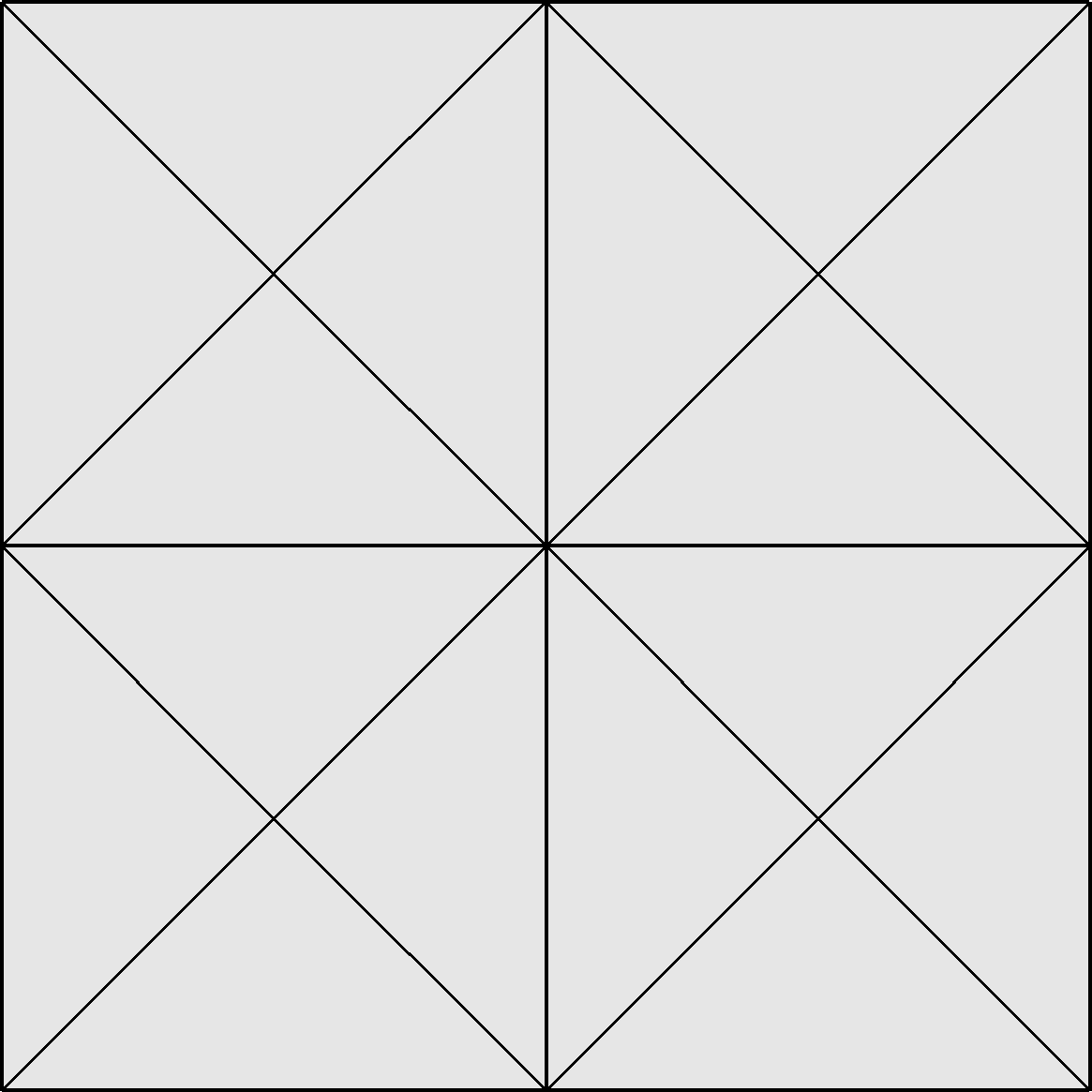} \\
  \includegraphics[width=1.2in]{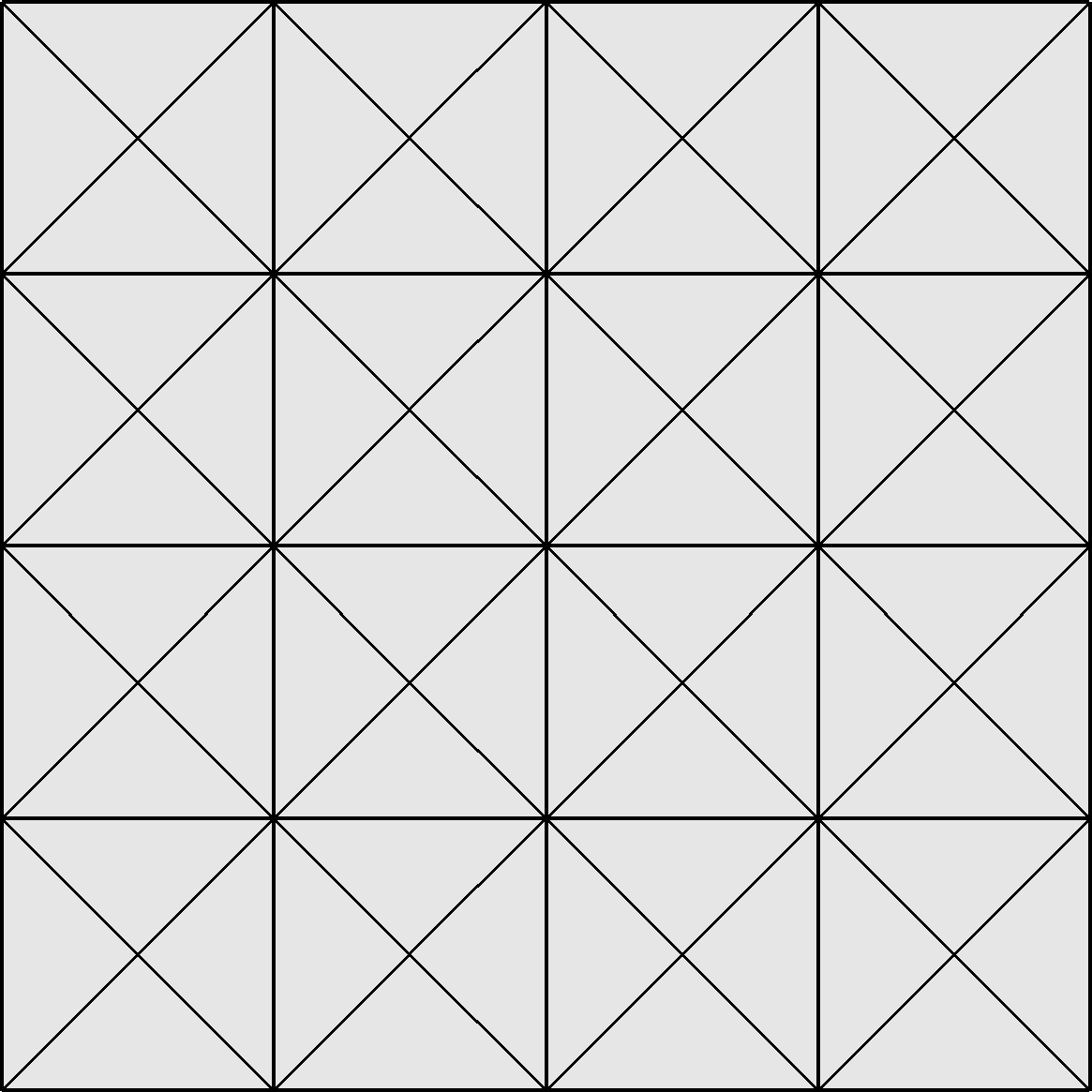} &
  \includegraphics[width=1.2in]{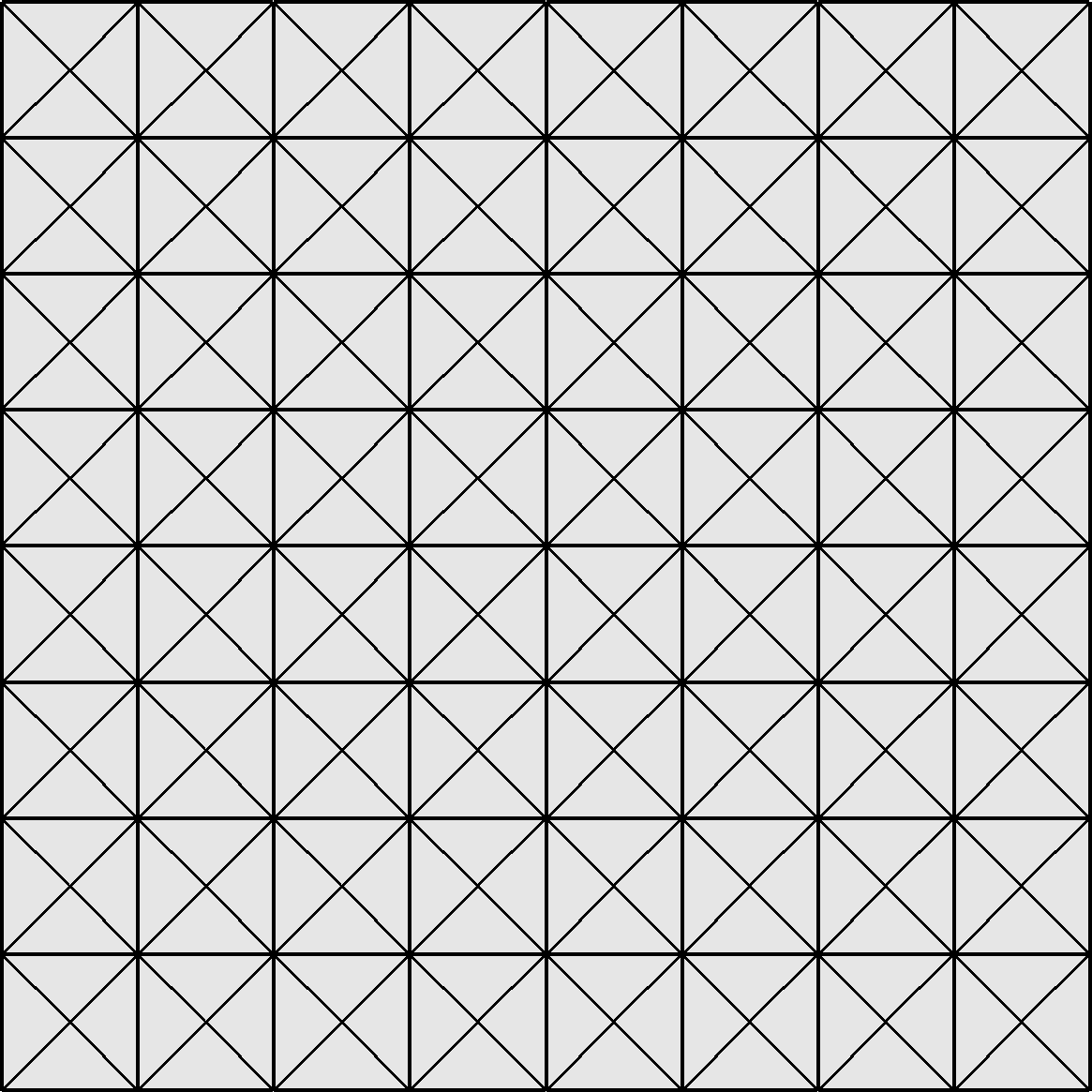} 
 \end{tabular}}
\caption{$\T_2$, $\T_1$, $\T_{1/2}$, $\T_{1/4}$, the first four crisscross triangulations.}\label{f:crisscross}
\end{figure}

Define $p:M\to\R$ by $p(x,y)=x-x^3/3$ and let $u=dp=(1-x^2)dx\in \Lambda^1(M)$.
Now for $q\in H\Lambda^0(M)$ (i.e., the Sobolev space $H^1(M)$), we have
$$
\star dq = \star\left(\frac{\partial q}{\partial x} dx + \frac{\partial q}{\partial y} dy\right)
= \frac{\partial q}{\partial x} dy - \frac{\partial q}{\partial y} dx,
$$
so
$$
\< u, dq\>_{L^2\Lambda^1}=\int_M u\wedge\star dq = \int_M (1-x^2)\frac{\partial q}{\partial x}\,dx\,dy 
= \int_M 2x q\,dx \,dy = \<2x,q\>_{L^2\Lambda^0}.
$$
Thus $u$ belongs to the domain of $d^*$ and
$d^*u = 2x$.  As an alternative verification, we may
identify $1$-forms and vector fields.  Then $u$ corresponds to the vector field $(1-x^2,0)$
which has vanishing normal component on $\partial M$, and so belongs to the domain of $d^*=-\div$
and $d^*u = -\div(1-x^2,0)=2x$.

Set $w_h =d_h^*\pi_hu.$  Now $w_h\in\Lambda^0_h$, i.e., it is a continuous piecewise linear
function.  The projections $\pi_h$ into the Whitney forms form a cochain map,
so $\pi_hu = \pi_h dp = d \pi_hp=\grad \pi_hp$, where $\pi_h p$ is piecewise linear interpolant
of $p$.  Thus $w_h\in \Lambda^0_h$ is determined by the equations
\begin{equation}\label{dlap}
\int_M w_h q\,dx\,dy = \int_M \grad \pi_h p\cdot \grad q\, dx\,dy, \quad q\in \Lambda^0_h.
\end{equation}
It turns out that we can give the solution to this problem explicitly.
Since $w_h$ is a continuous piecewise linear function, it
is determined by its values at the vertices of the triangulation $\T_h$.  The coordinates
of the vertices are integer multiples of $h/2$.  In fact the value of $w_h$ at
a vertex $(x,y)$ depends only on $x$ and for $h\le 1$ is given by
$$
w_h(x,y) = \begin{cases}
  -h, & x=-1, \\
  0, & -1<x<1, \ \text{$x$ a multiple of $h$},\\
  h, & x = 1,\\
  -6+2h, & x=-1+h/2,\\
  6x, & -1 + h/2 < x < 1-h/2, \ \text{$x$ an odd multiple of $h/2$},\\
  6-2h, & x = 1-h/2.          
           \end{cases}
$$
A plot of the piecewise linear function
$w_h$ is shown in Figure~\ref{f:cplot} for $h = 1/2$.  To verify the formula
it suffices to check \eqref{dlap} for all piecewise linear functions $q$ that vanish on all vertices except one.  There are several cases depending on how close the vertex is to the boundary, and
the computation is tedious, but elementary. Here we only give the details when the vertex is $(x,y)$ with $-1+h/2 <x<1-h/2$ and $x$ is an odd multiple of $h/2$. To this end, let $q$ be the piecewise linear function that is one on vertex $(x,y)$  and vanishes on all the remaining vertices. In this case, the support of  $q$ is the union of the four triangles $T_1, T_2, T_3, T_4$ that have $(x,y)$ as a vertex (see Figure \ref{f:triangles}). According to the formula, in the support of $q$, one has $w_h= 6\, x \, q$. A simple calculation then shows that the left-hand side of \eqref{dlap} is 
\begin{equation*}
\int_M w_h q\,dx\,dy = 6\, x \sum_{i=1}^4 \int_{T_i} q^2 dx\, dy= 4\,x\,m,
\end{equation*}
where $m=h^2/4=|T_i|$ for any $i$.

To calculate the right-hand side of \eqref{dlap} for this $q$, we calculate that
$$
\grad q = \frac{2}{h} \begin{cases}
  (1, 0), & \text{ on } T_1, \\
   (0,1), & \text{ on } T_2,\\
   (-1, 0), & \text{ on } T_3,\\
   (0,-1), & \text{ on } T_4,
           \end{cases}
$$
and
$$
\grad \pi_h p= \frac{2}{h} \begin{cases}
  (p(x)-p(x-\frac{h}2), 0), & \text{ on } T_1, \\
   (\frac{1}{2}[p(x+\frac{h}{2})-p(x-\frac{h}{2})],p(x)-\frac{1}{2}[p(x+\frac{h}{2})+p(x-\frac{h}{2})]), & \text{ on } T_2,\\
   (p(x+\frac{h}{2})-p(x), 0), & \text{ on } T_3,\\
   (\frac{1}{2}[p(x+\frac{h}{2})-p(x-\frac{h}{2})],\frac{1}{2}[p(x+\frac{h}{2})+p(x-\frac{h}{2})]-p(x)), & \text{ on } T_4.
           \end{cases}
$$
Hence, 

\begin{equation*}
\begin{split}
\int_M \grad \pi_h p\cdot \grad q\, dx\,dy 
&= \sum_{i=1}^4 \int_{T_i} \pi_h p\cdot \grad q\, dx\,dy \\
&= \frac{16} {h^2} (p(x)-\frac{1}{2}[p(x-\frac{h}{2})+p(x+\frac{h}{2})]) m
=4 \, x m.
\end{split}
\end{equation*}
This verifies \eqref{dlap} for this piecewise linear function $q$.  

\begin{figure}[htb]
\centerline{%
\setlength{\unitlength}{.8in}
\begin{picture}(2.4,2.4)
 \linethickness{.5pt}
 \put(0, 0){\line(1, 0){2}}
 \put(0, 0){\line(0, 1){2}}
 \qbezier(0, 0)(1,1)(2,2)
 \qbezier(2, 0)(1,1)(0,2)
 \put(2, 0){\line(0, 1){2}}
 \put(0, 2){\line(1, 0){2}}
 \put(1.1, .95){$(x,y)$}
 \put(-.8,-.2){$(x-h/2,y-h/2)$}
 \put(1.2,2.1){$(x+h/2,y+h/2)$}
 \put(.3, .9){$T_1$}
 \put(.9, .3){$T_2$}
 \put(1.7, .9){$T_3$}
 \put(.9, 1.6){$T_4$}
\end{picture}}
\vspace{10pt}
\caption{The support of the piecewise linear function $q$.}\label{f:triangles}
\end{figure}
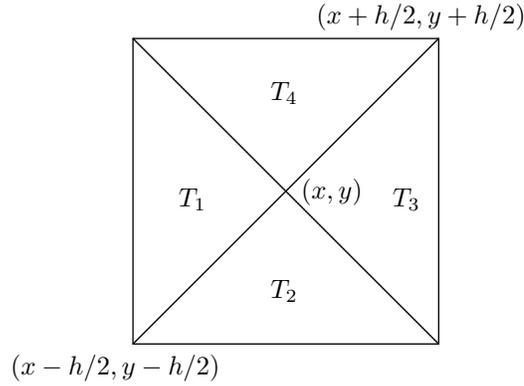

\begin{figure}[htb]
\centerline{%
 \begin{tabular}{cccc}
  \includegraphics[width=3in]{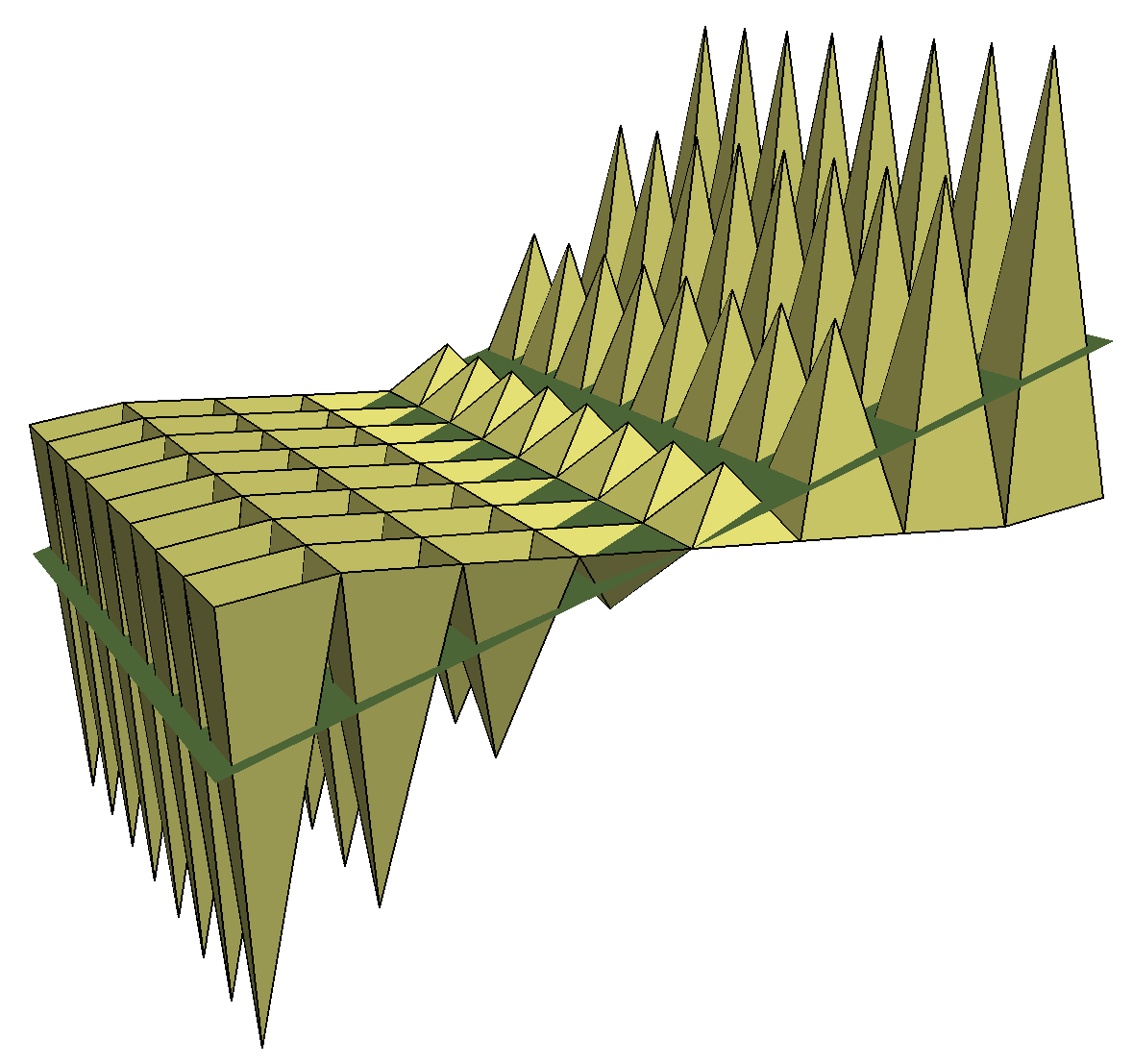} 
 \end{tabular}}
\caption{The spiked surface is the graph of the piecewise linear function $w_h=d^*_h \pi_h f$ for $h=1/4$.
The plane is the graph of the linear
function $d^*u$.}\label{f:cplot}
\end{figure}

Finally, we note that, since
$w_h$ essentially oscillates between $6x$ and $0$, it does not converge in $L^2$ to  $d^*u$ (or to anything
else) as $h$ tends to zero.

\section{Consistency for $1$-forms on piecewise uniform meshes}\label{s:superconv}
We continue to consider a sequence of triangulations $\T_h$ indexed by a positive parameter
$h$ tending to $0$.  We take $h$ to be equivalent to the maximal
simplex diameter
$$
c h \le \max_{T\in\Delta_n(\T_h)}\diam T \le C h,
$$
for some positive constants $C,c$ independent of $h$ (throughout we denote by $C$ and $c$
generic constants, not necessarily the same in different occurrences).  We also assume that the sequence of triangulations is \emph{shape
regular} in the sense that there exists $c>0$ such that
$$
\rho(T) \ge c \diam T,
$$
for all $T\in\T_h$ and all $h$, 
where $\rho(T)$ is the diameter of the ball inscribed in $T$.

We begin with some estimates for the approximation of a $k$-form by an element of $\Lambda^k_h$.
For this we need to introduce the spaces of differential forms with coefficients in a Sobolev space.
Let $m$ be a non-negative integer and $u$ a $k$-form defined on a domain $M\subset\R^n$, which
we may expand as
\begin{equation}\label{kform}
u= \sum_{1\le i_1<\cdots<i_k\le n} u_{i_1\cdots i_k}\,dx^{i_1}\wedge\cdots\wedge dx^{i_k}.
\end{equation}
Using multi-index notation for partial derivatives of the coefficients $u_{i_1\cdots i_k}$,
we define the $m$th Sobolev norm and seminorm by
\begin{align*}
\|u\|_{H^m\Lambda^k}^2 &= \sum_{1\le i_1<\cdots<i_k\le n}\ \sum_{|\alpha|\le m} \|D^\alpha u_{i_1\cdots i_k}\|_{L^2(M)}^2,
\\
|u|_{H^m\Lambda^k}^2 &= \sum_{1\le i_1<\cdots<i_k\le n}\ \sum_{|\alpha|= m} \|D^\alpha u_{i_1\cdots i_k}\|_{L^2(M)}^2,
\end{align*}
and define the space $H^m\Lambda^k(M)$ to consist of all $k$-forms in $M$ for which
the Sobolev norm $\|u\|_{H^m\Lambda^k}$ is finite.

With this notation, we can state the basic approximation result that for any
shape regular sequence of triangulations there is a constant $C$ such that
\begin{equation}\label{e:H1approx}
\inf_{v\in \Lambda^k_h} \|u-v\| \le C h \|u\|_{H^1\Lambda^k}, \quad u\in H^1\Lambda^k(M).
\end{equation}
For a proof, see \cite[Theorem~5.8]{afw-bull}.  Since $H^1\Lambda^k$ is dense in $L^2\Lambda^k$,
this implies that
\begin{equation}\label{approx}
\dist(f,\Lambda^k_h) := \inf_{v\in \Lambda^k_h}\|f-v\| \to 0  \text{ as $h\to0$}, \quad f\in L^2\Lambda^k(M).
\end{equation}

In addition to the best approximation estimate \eqref{e:H1approx}, we also need
an $O(h)$ estimate on the projection error $\|u-\pi_h u\|$.   For this we require
more regularity of $u$, since $\pi_h u$ is defined in terms of traces of $u$ on $k$-dimensional faces, which
need not be defined on $H^1\Lambda^k$.
\begin{lem}
Let $\{\T_h\}$ be a shape regular sequence of triangulations of $M\subset\R^n$ and $k$ an integer between
$0$ and $n$.  Let $\ell$ be the smallest integer so that $\ell>(n-k)/2$. Then there exists a constant $C$,
depending only on $n$ and the shape regularity constant, such that
\begin{equation}\label{proj-estimate}
\|\pi_h u-u\|_{L^2\Lambda^k} \le C \, \sum_{m=1}^{\ell} h^m|u|_{H^m\Lambda^k}, \quad u\in H^\ell\Lambda^k(M).
\end{equation} 
\end{lem}
\begin{proof}
First we note that
the canonical projection is defined simplex by simplex, as
$$
(\pi_h u)|_T = \pi_T(u|_T),
$$
where, for $v$ a $k$-form on $T$, $\pi_T v$ is its interpolant into the space of Whitney forms
on the single simplex $T$.  Therefore, it is enough to prove that
\begin{equation}\label{etp}
\|u-\pi_T u\|_{L^2\Lambda^k(T)} \le C \, \sum_{m=1}^{\ell} h^m|u|_{H^m\Lambda^k(T)}, \quad u\in H^\ell\Lambda^k(T),
\end{equation}
with the constant $C$ depending on $T$ only through its shape constant.
We prove this first for the unit right simplex in $\R^n$, $\hat T$, with vertices at the origin and
the $n$ points $(1,0,\ldots,0)$, $(0,1,0,\ldots)$, \dots.  Since $\ell>(n-k)/2$, we obtain,
by the Sobolev embedding theorem, that
$\|\pi_{\hat T} u\|_{L^2\Lambda^k(\hat T)}\le C\|u\|_{H^\ell\Lambda^k(\hat T)}$, and so, by the triangle inequality,
$$
\|u-\pi_{\hat T} u\|_{L^2\Lambda^k(\hat T)}\le C\|u\|_{H^\ell\Lambda^k(\hat T)}.
$$
Now let $\bar u = n!\int_{\hat T}u$, a constant $k$-form on $\hat T$ equal to the average
of $u$.  Then $\pi_{\hat T}\bar u = \bar u$, so
\begin{multline*}
\|u-\pi_{\hat T} u\|_{L^2\Lambda^k(\hat T)}
= \|(u-\bar u)-\pi_{\hat T} (u-\bar u)\|_{L^2\Lambda^k(\hat T)}
\\
\le C \|u-\bar u\|_{H^\ell\Lambda^k(\hat T)} \le C(\|u-\bar u\|_{L^2\Lambda^k(\hat T)} 
+\sum_{m=1}^\ell |u|_{H^m\Lambda^k(\hat T)}),
\end{multline*}
where we have used the fact that $\bar u$ is a constant form, so its $m$th Sobolev
seminorm vanishes for $m\ge 1$.
Now we invoke Poincar\'e's inequality
$$
\|u-\bar u\|_{L^2\Lambda^k(\hat T)}\le C|u|_{H^1\Lambda^k(\hat T)}.
$$
Putting things together, and writing $\hat u$ instead of $u$, we have shown that
\begin{equation}\label{hat}
\|\hat u-\pi_{\hat T} \hat u\|_{L^2\Lambda^k(\hat T)}\le C\sum_{m=1}^\ell |\hat u|_{H^m\Lambda^k(\hat T)}, \quad
\hat u \in H^\ell\Lambda^k(\hat T).
\end{equation}
This is the desired result \eqref{etp} in the case $T=\hat T$.

To obtain the result for a general simplex, we scale via an affine diffeomorphism $F:\hat T\to T$.
If $u$ is the $k$-form on $T$ given by \eqref{kform}, then
\begin{equation}\label{pullback}
F^* u = \sum_{\{1\le i_1<\cdots <i_k\le n\}}\ \sum_{j_1,\ldots,j_k=1}^n 
  (u_{i_1\cdots i_k}\circ F) \pd{F^{i_1}}{\hat x^{j_1}}\cdots\pd{F^{i_k}}{\hat x^{j_k}}\,
  d\hat x^{j_1}\wedge\cdots\wedge d\hat x^{j_k}.
\end{equation}
Each of the partial derivatives $\partial F^{i_p}/\partial \hat x^{j_q}$ is a constant
bounded by $h$.  Using the chain rule and change of variables in the integration, we
find that
\begin{equation}\label{scaling}
c|F^*u|_{H^m\Lambda^k(\hat T)} \le (\operatorname{vol}{T})^{-1/2}
\, h^{m+k} |u|_{H^m\Lambda^k(T)}
\le C|F^*u|_{H^m\Lambda^k(\hat T)},
\end{equation}
where the constants $c$ and $C$ depend only on $m$ and $n$ and the shape regularity constant of $T$.
Combining \eqref{hat} and \eqref{scaling} we get
\begin{multline*}
\|u-\pi_T u\|_{L^2\Lambda^k(T)}
\le C(\operatorname{vol}{T})^{1/2} h^{-k}\|\hat u - \pi_{\hat T}\hat u\|_{L^2\Lambda^k(\hat T)}
\\
\le C (\operatorname{vol}{T})^{1/2}h^{-k}\sum_{m=1}^\ell |\hat u|_{H^m\Lambda^k(\hat T)}
\le C\sum_{m=1}^\ell h^m|u|_{H^m\Lambda^k(T)},
\end{multline*}
which establishes \eqref{etp}.
\end{proof}

Our approach to bounding the norm of the consistency error is to relate it to another quantity
which has been studied in the finite element literature, namely
\begin{equation}\label{defA}
A_h(u):= \sup_{v_h\in \Lambda_h^{k-1}}\frac{\langle u-\pi_hu, \exd v_h \rangle}{\|v_h\|}.
\end{equation}
\begin{thm}\label{t:equiv}
Assume the approximation property \eqref{approx}.
Then, for any smooth $u\in L^2\Lambda^k$ belonging to the domain of $d^*$ we have
$$
\lim_h \|d^*u-d^*_h\pi_hu\| = 0 \iff \lim_h A_h(u) =0.
$$
\end{thm}
This follows immediately from Lemma~\ref{t:eq}.
\begin{lem}\label{t:eq}
Let $1\leq k\leq n$, and let $u\in L^2\Lambda^k$ be smooth and in the domain of $d^*$.
Then
\begin{equation}\label{e:eq}
A_h(u)
\leq
\|\exd^*u-\exd^*_h\pi_hu\| 
\leq 
\dist(\exd^*u,\Lambda_h^{k-1})+A_h(u).
\end{equation}
\end{lem}

\begin{proof}
The first inequality is straightforward.  For any $v_h\in \Lambda^{k-1}_h$, 
\begin{equation*}
\frac{\langle u-\pi_hu, \exd v_h \rangle}{\|v_h\|}
=\frac{\langle \exd^*u-\exd^*_h\pi_hu,v_h\rangle}{\|v_h\|}
\le\|\exd^*u-\exd^*_h\pi_hu\|.
\end{equation*}
For the second inequality, we introduce the $L^2$-orthogonal projection
$P_h:L^2\Lambda^{k-1}\to \Lambda_h^{k-1}$
and invoke the triangle inequality to get
\begin{equation}\label{e:triangle}
\|\exd^*u-\exd^*_h\pi_hu\| \leq \|\exd^*u-P_h\exd^*u\|+\|P_h\exd^*u-\exd^*_h\pi_hu\|
=\dist(d^*u,\Lambda^{k-1}_h)+\|w\|,
\end{equation}
where $w=P_h\exd^*u-\exd^*_h\pi_hu\in\Lambda^k_h$.  Now
\begin{equation}
\|w\|^2 = \langle P_h\exd^*u-\exd^*_h\pi_hu,w\rangle=\langle u-\pi_hu,\exd w\rangle,
\end{equation}
and hence
\begin{equation}
\|w\| = \frac{\langle u-\pi_hu,\exd w\rangle}{\|w\|}
\leq\sup_{v_h\in \Lambda_h^{k-1}}\frac{\langle u-\pi_hu, \exd v_h \rangle}{\|v_h\|}=A_h(u),
\end{equation}
which completes the proof.
\end{proof}

Thus we wish to bound $\langle u-\pi_h u,d v_h\rangle/\|v_h\|$ for smooth $u$ in the domain
of $d^*$ and $v_h\in\Lambda^k_h$.  An obvious approach is to apply the Cauchy--Schwarz inequality
and then use the approximation estimate \eqref{proj-estimate} to obtain
\begin{equation}\label{cs}
|\langle u-\pi_h u,d v_h\rangle|\le \|u-\pi_h u\|\|d v_h\| \le Ch\|u\|_{H^\ell \Lambda^k}\|d v_h\|.
\end{equation}
To continue, we need to bound $\|d v_h\|/\|v_h\|$ for $v_h$ an arbitrary non-zero element of $\Lambda^k_h$.
Because $\Lambda^k_h$ consists of piecewise polynomials, it is possible to bound its derivative in
terms of its value using a Bernstein type inequality or inverse estimate.  This gives that
\begin{equation}\label{e:inv}
\|d v_h\|\le C\underline h^{-1} \|v_h\|, \quad v_h\in\Lambda^k_h,
\end{equation}
where $\underline h = \min_{T\in\Delta_n(\T_h)} \diam T$.  Unfortunately, even if
we assume that our triangulations
are \emph{quasiuniform}, i.e., that $\underline h \ge c h$ for some fixed $c>0$, this just leads to
the bound
$$
A_h(u) \le C\|u\|_{H^\ell\Lambda^k},
$$
which does not tend to zero with $h$.  In fact, we cannot hope to get a bound
which tends to zero without further hypotheses, since, as we have seen,
even for the nice mesh sequence and form $u$ considered in the previous section, $d^*_h$ is not consistent,
and so $A_h(u)$ does not tend to zero.

Nonetheless, for very special mesh sequences it is possible to improve the bound
\eqref{cs} from first to second order in $h$.  This was established by
Brandts and K\v r\'\i\v zek in their work on gradient superconvergence \cite{BK03}.
The mesh condition is embodied by the following concept.

\begin{defn}[\cite{BK03}]
A triangulation $\T$ on $M$ is called \emph{uniform}
if there exist $n$ linearly independent vectors $e_1,\ldots,e_n$, such that
\begin{enumerate}
\item
Every simplex in $\T$ contains an edge parallel to each $e_j$.
\item
If an edge $e$ is parallel to one of the $e_j$ and is not contained in $\partial M$,
then the union $P_e$ of simplices containing $e$ is invariant under reflection
through the midpoint $m_e$ of $e$,
i.e., $2m_e-x\in P_e$ for all $x\in P_e$.
\end{enumerate}
\end{defn}

The crisscross triangulations shown in Figure~\ref{f:crisscross} satisfy the first condition
of the definition, but not the second, and so are not uniform.  On the other hand, the mesh
sequence that is obtained by starting from a single triangle, or from a division of a square into
two triangles and applying regular standard subdivision, is uniform.  See the
first two rows of Figure~\ref{f:uniform}.  A
uniform triangulation of the cube in $n$ dimensions is obtained by subdividing it into
$m^n$ subcubes, and dividing each of these into $n!$ simplices sharing a common diagonal,
with all the diagonals of the subcubes chosen to be parallel.  The 3D case is shown
in Figure~\ref{f:uniform}. 
We refer to \cite{BK03} for more details.

\begin{figure}[htb]
\centerline{%
 \begin{tabular}{cccc}
  \includegraphics[width=1.1in]{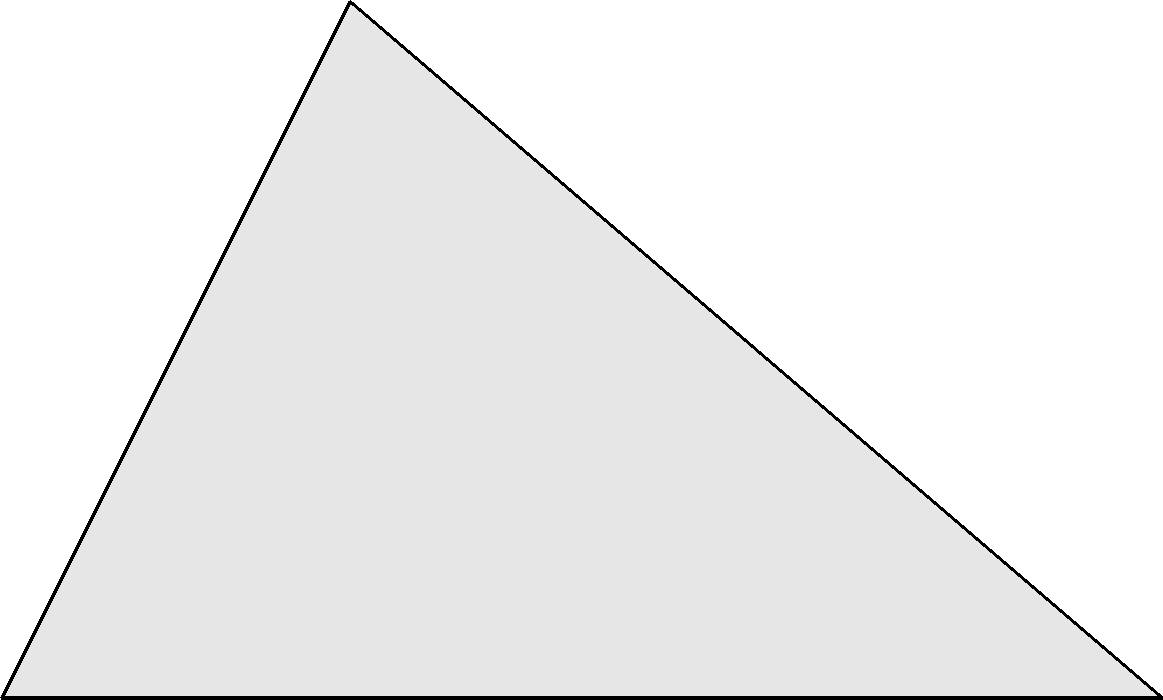} &
  \includegraphics[width=1.1in]{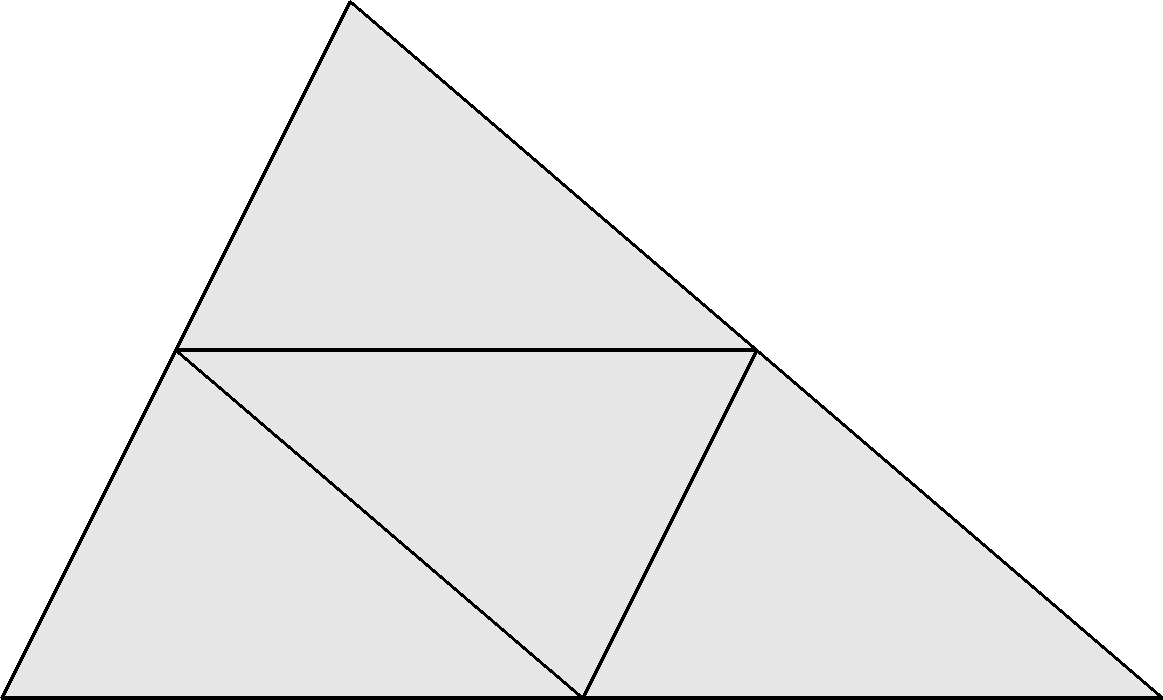} &
  \includegraphics[width=1.1in]{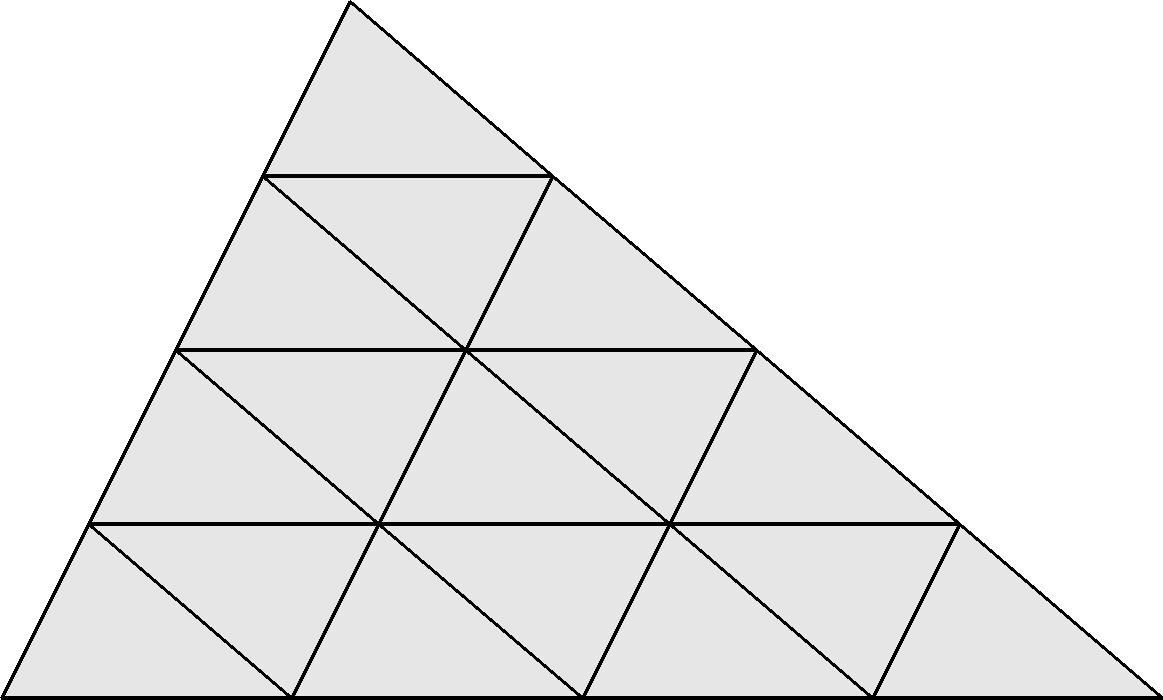} &
  \includegraphics[width=1.1in]{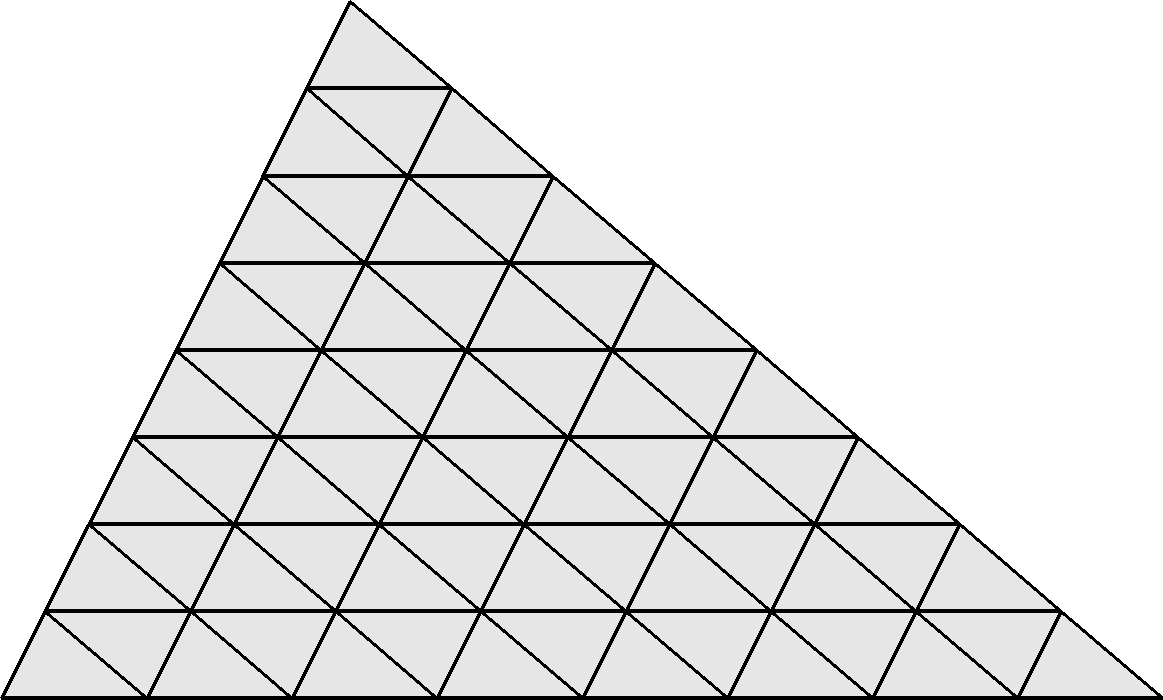} \\[1ex]
  \includegraphics[width=1.1in]{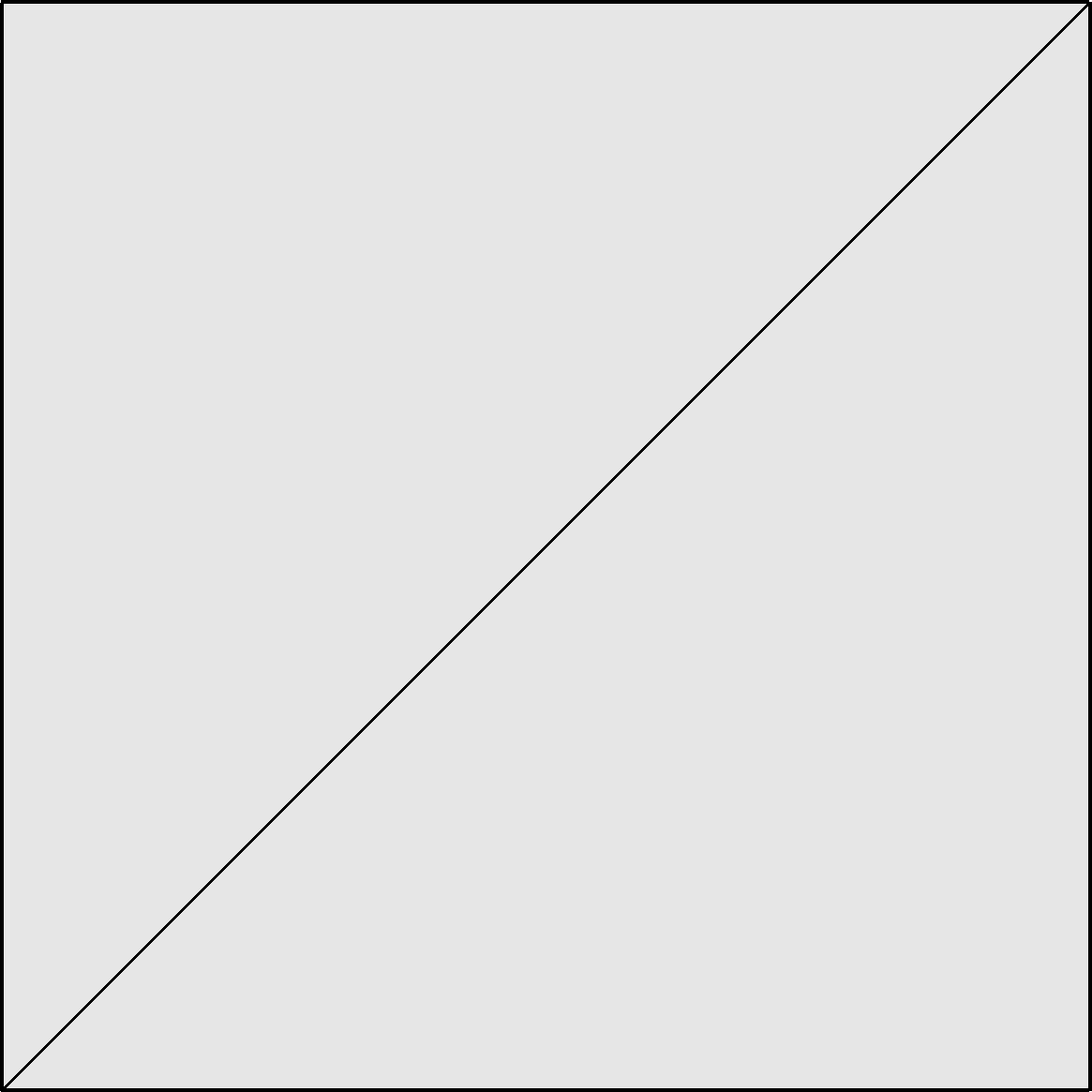} &
  \includegraphics[width=1.1in]{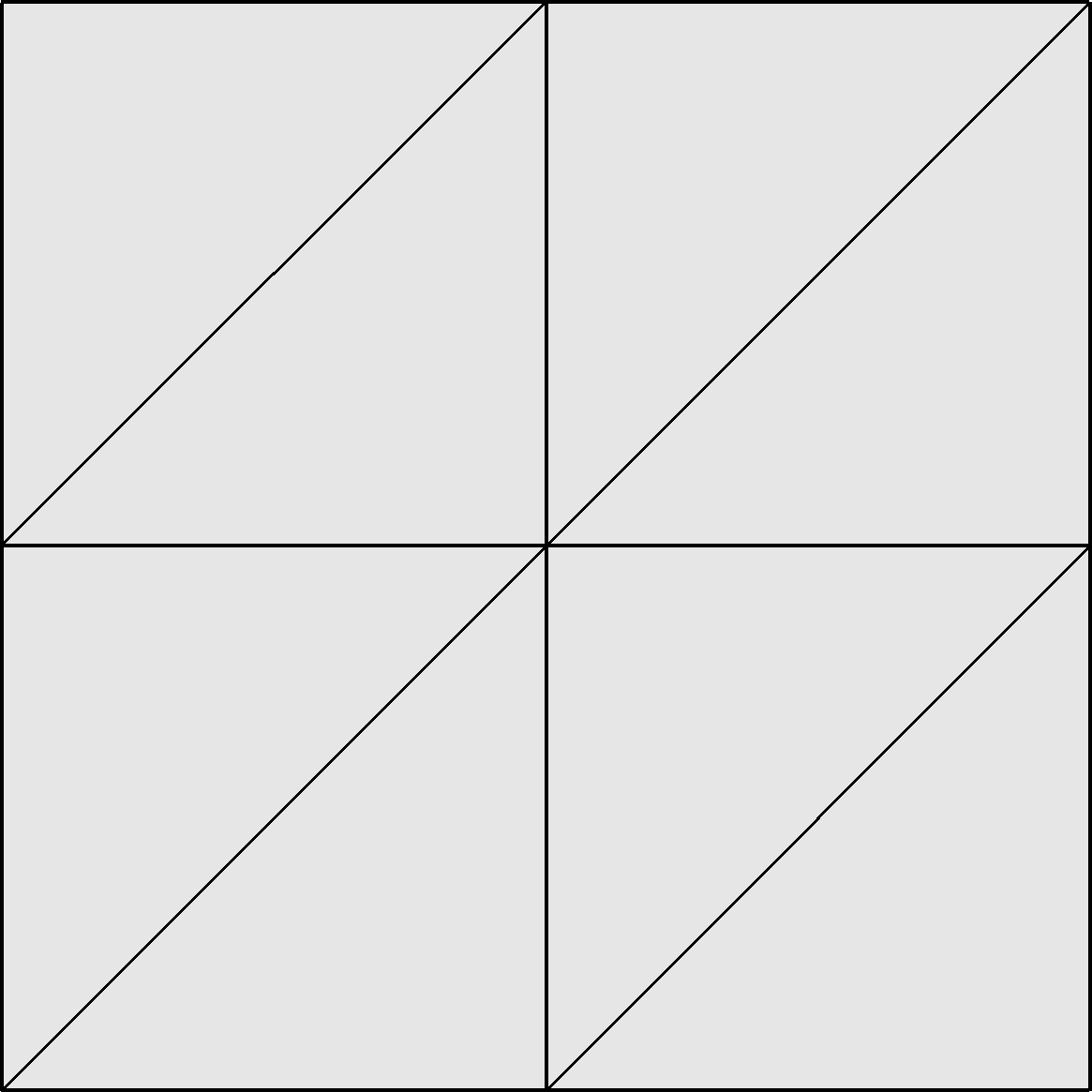} &
  \includegraphics[width=1.1in]{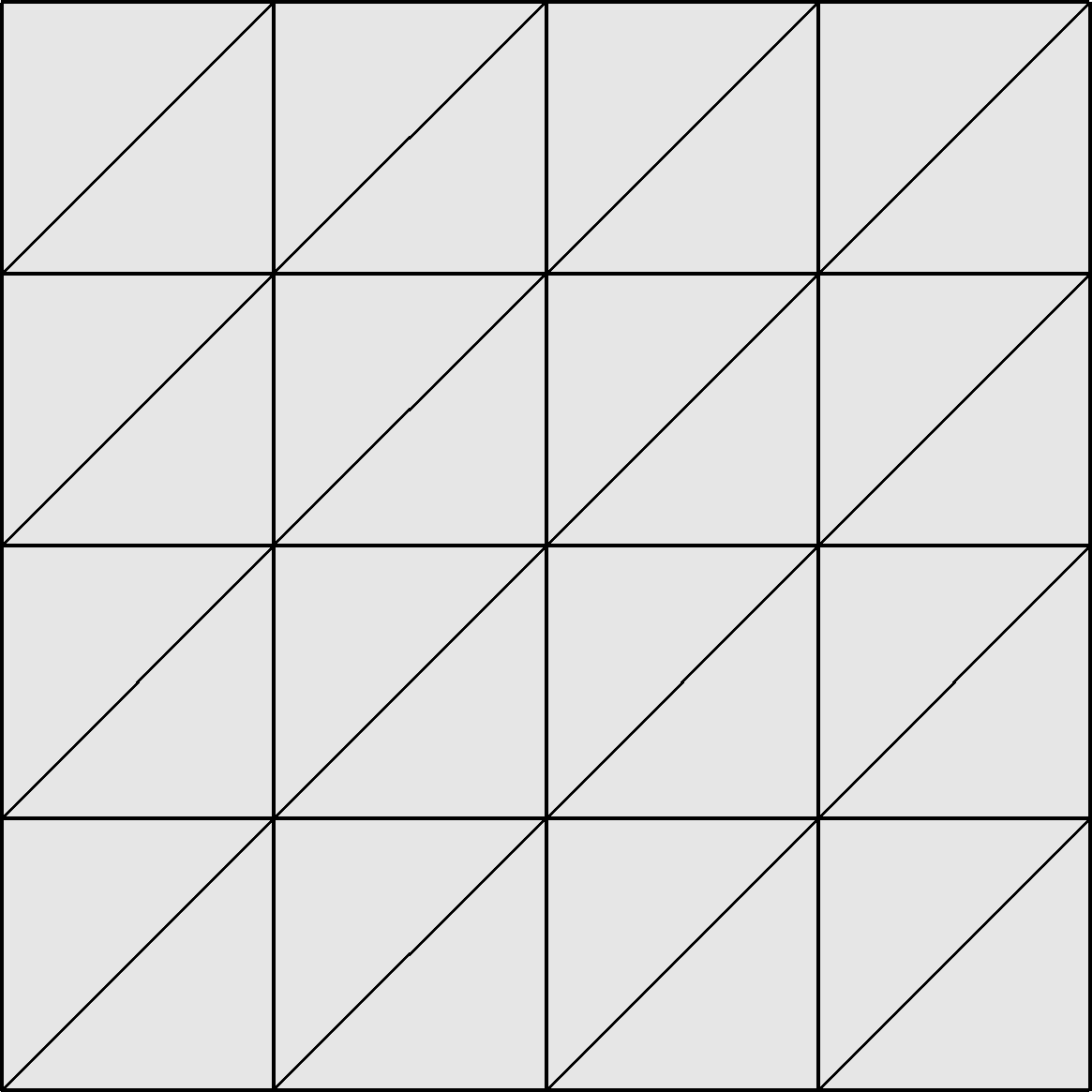} &
  \includegraphics[width=1.1in]{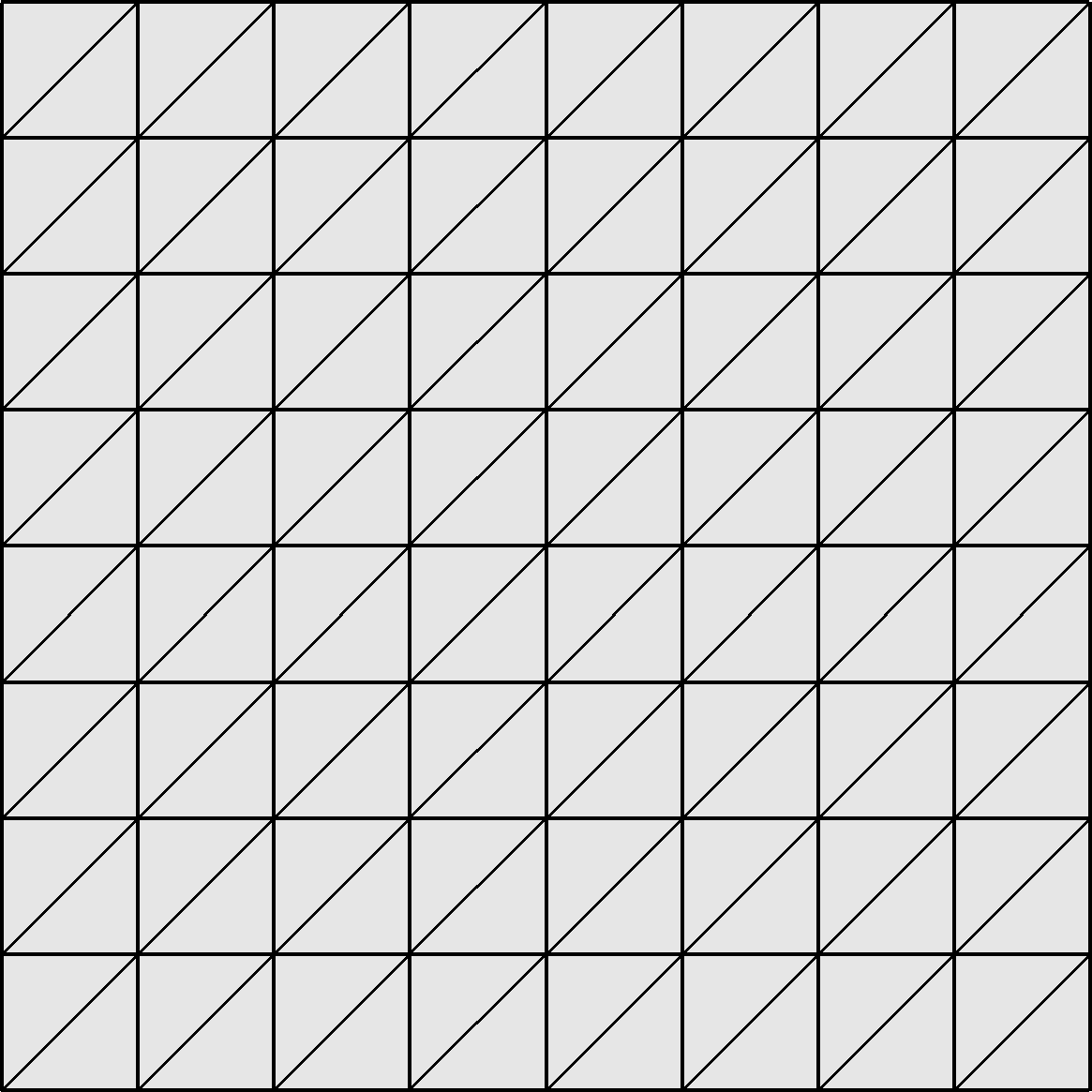} \\[1ex]
  \includegraphics[width=1.1in]{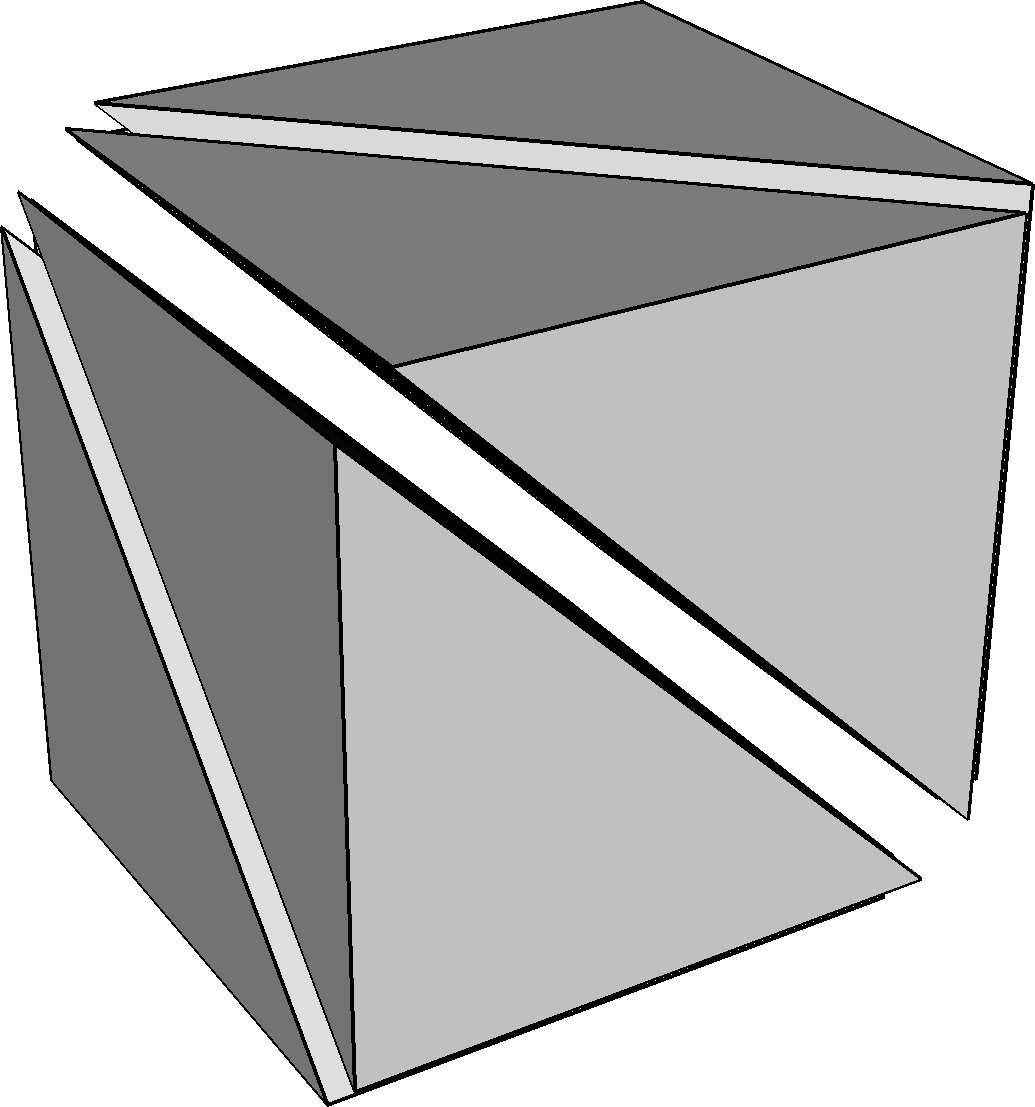} &
  \includegraphics[width=1.1in]{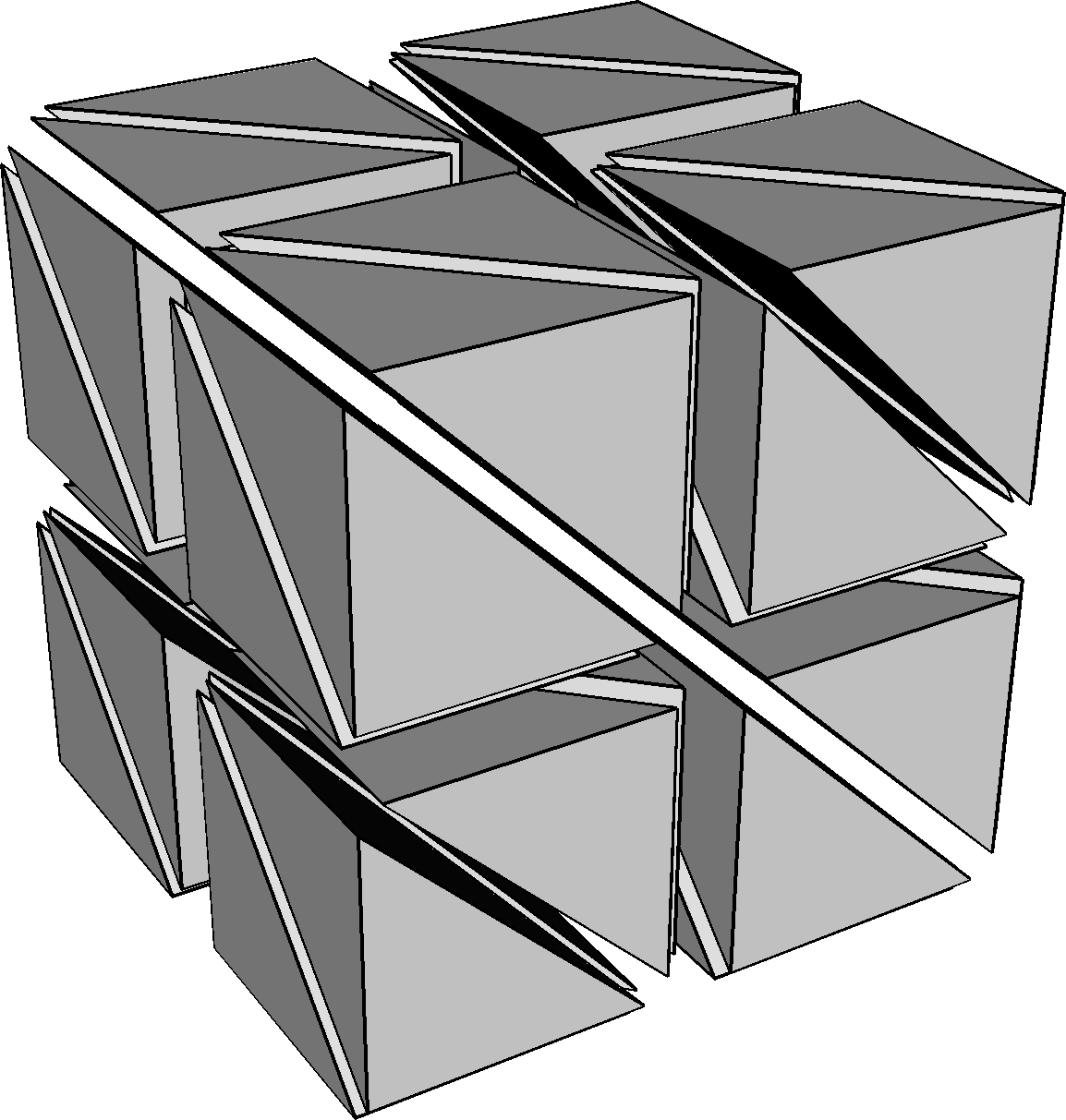} &
  \includegraphics[width=1.1in]{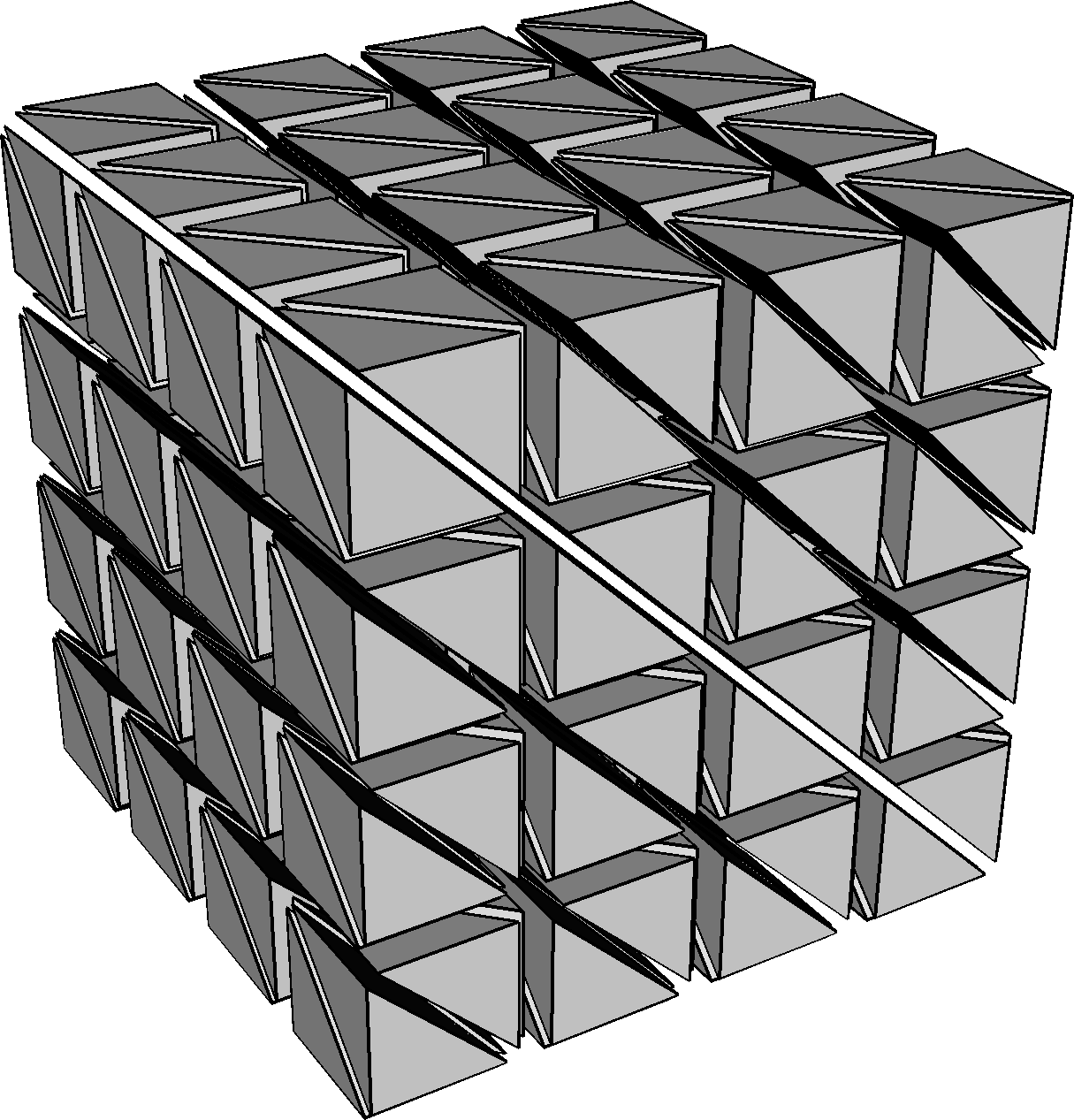} &
  \includegraphics[width=1.1in]{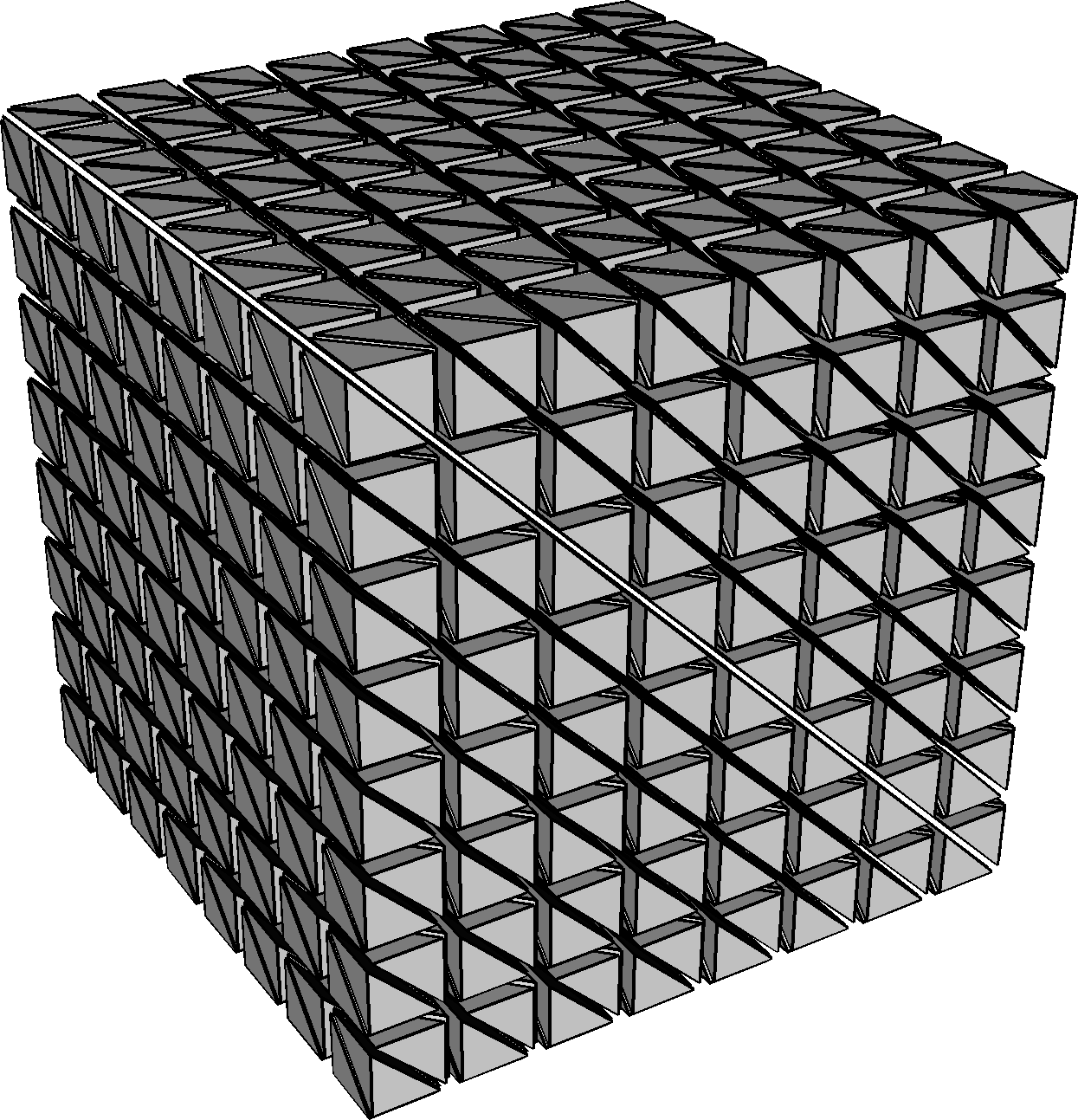}
 \end{tabular}}
\caption{Uniform triangulations.}\label{f:uniform}
\end{figure}

\begin{figure}[htb]
\centerline{%
 \begin{tabular}{cc}
  \includegraphics[width=1.2in]{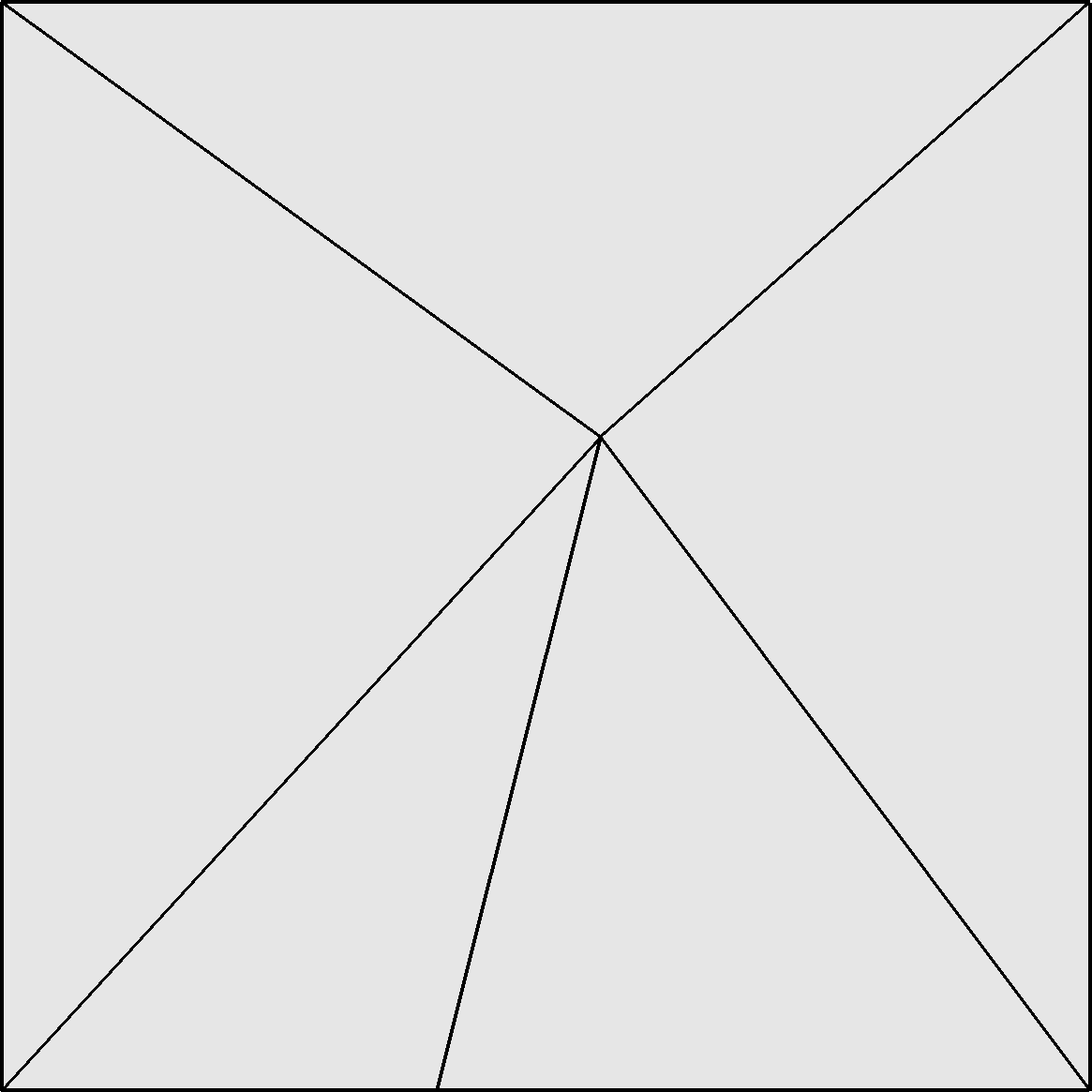} &
  \includegraphics[width=1.2in]{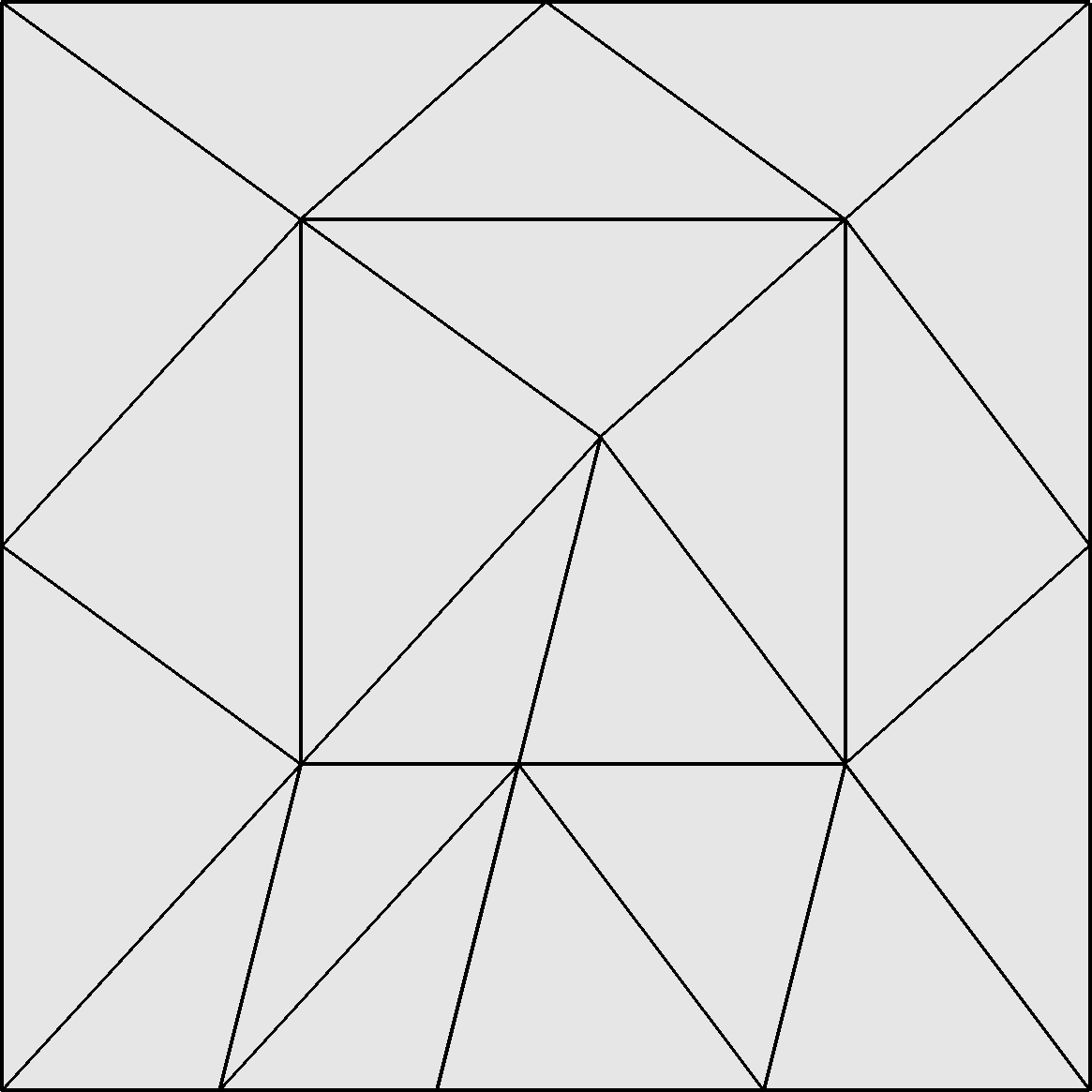} \\
  \includegraphics[width=1.2in]{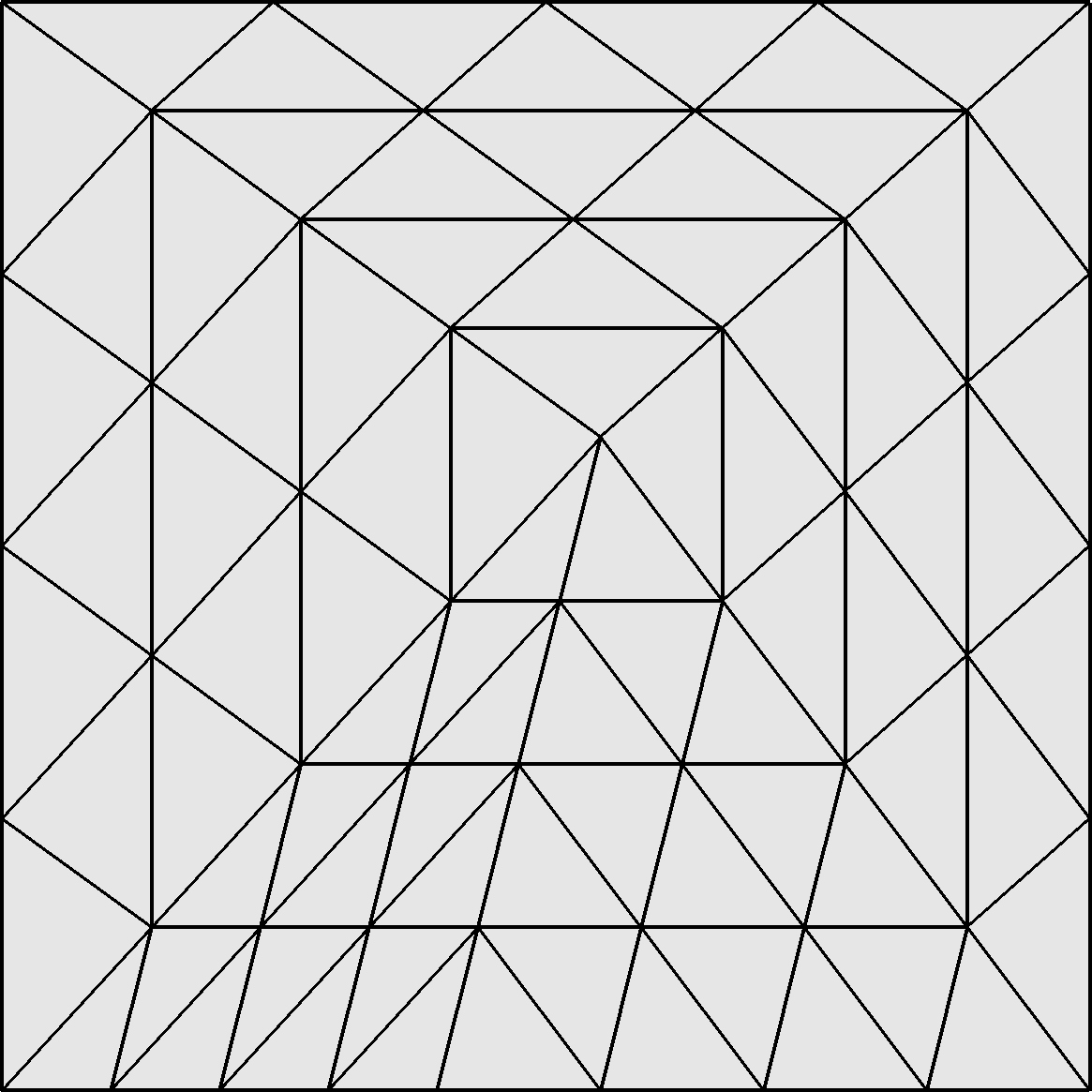} &
  \includegraphics[width=1.2in]{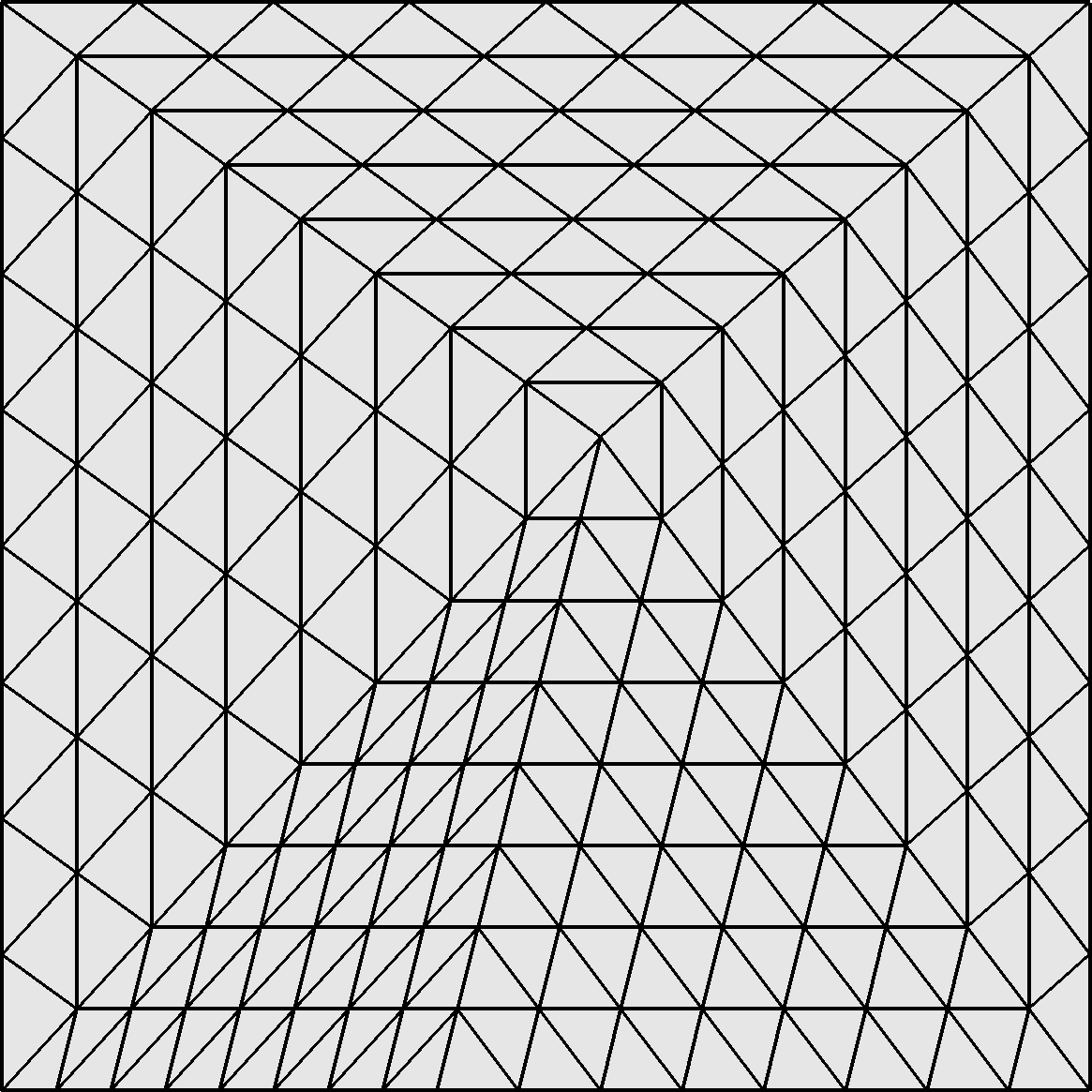}
 \end{tabular}}
\caption{A piecewise uniform sequence of triangulations.}\label{f:puniform}
\end{figure}

Theorem 3.4 of \cite{BK03} claims that if $\{\T_h\}$ is a shape regular family of uniform triangulations of $M$,
and if $u$ is a smooth $1$-form,
then there exists a constant $C>0$ such that
\begin{equation*}
|\langle \pi_hu-u, \exd v_h \rangle| \leq C h^2\|u\|_{H^2\Lambda^1}\|\exd v_h\|,
\end{equation*}
for all $v_h\in \Lambda_h^0\cap \ring H^1(M)$ and $h>0$.
Here $\ring H^1(M)$ denotes the space of $H^1(M)$ functions with vanishing trace on $\partial M$.
However, their proof uses the inequality (cf. (1.5) of \cite{BK03})
\begin{equation*}
\|\pi_hu-u\|_{L^2\Lambda^1} \le C \, h \|u\|_{H^1\Lambda^1},
\end{equation*}
where $C$ is a constant independent of $u$.
This would imply that $\pi_h$ can be continuously extended to $H^1\Lambda^1$, which is impossible for $n\geq3$.
Fortunately, the proof in \cite{BK03} works verbatim if the above inequality is replaced by \eqref{proj-estimate}.
Hence, the following result is essentially proved in \cite{BK03}.

\begin{thm} \label{t:sup-con}
Let $\{\T_h\}$ be a shape regular family of uniform triangulations of $M$,
and let $u$ be a smooth $1$-form. Furthermore, let $\ell$ be the smallest integer so that $\ell>(n-1)/2$.
Then there exists a constant $C>0$ such that 
\begin{equation}\label{e:h2-bd}
|\langle \pi_hu-u, \exd v_h \rangle| \leq C h^2\|u\|_{H^\ell\Lambda^1}\|\exd v_h\|,
\end{equation}
for all $v_h\in \Lambda_h^0\cap \ring H^1(M)$ and $1>h>0$.
\end{thm}

Next we consider piecewise uniform sequences of triangulations.
\begin{defn}
A family $\T_h$ of triangulations of the polytope $M$ is called
\emph{piecewise uniform} if there is a triangulation $\T$ of $M$
such that for each $h$, $\T_h$ is a refinement of $\T$ and for each
$T\in\T$ and each $h$, the restriction
of $\T_h$ to $T\in\Delta_n(T)$ is uniform.
\end{defn}
If, as in \cite{smits}, we start with an arbitrary triangulation of a polygon and refine it
by standard regular subdivision, the resulting sequence of triangulations
is piecewise uniform.  This is illustrated in Figure~\ref{f:puniform}.
The following theorem shows that $d^*$ is consistent for $1$-forms on 
piecewise uniform meshes, thus generalizing the main result of \cite{smits}
from $2$ to $n$ dimensions.
\begin{thm}\label{t:puniform}
Assume that the family of triangulations
$\{\T_h\}$ is a shape regular, quasiuniform, and piecewise uniform.
Let $u\in H^\ell \Lambda^1(M)$ be a $1$-form in the domain of $d^*$, where $\ell$ is the smallest integer satisfying $\ell >(n-1)/2$.
Then we have
\begin{equation}
\lim_{h\to 0}\|\exd^*u-\exd^*_h\pi_h u\| = 0.
\end{equation}
\end{thm}

\begin{proof}
Let $\T$ denote the triangulation of $M$ with respect to which the triangulations $\T_h$ are uniform.
We will apply Theorem \ref{t:sup-con} to the uniform mesh sequences obtained by
restricting $\T_h$ to each $T\in\T$.  To this end, let $K=\bigcup_{T\in\T}\partial T$ denote the
skeleton of $\T$,
and set
$$
\Sigma_h = \bigcup \{\,T\in \T_h\,|\,T\cap K\ne\emptyset\,\}.
$$
We can decompose an arbitrary function $v_h\in \Lambda_h^0$ as
\begin{equation}\label{e:decomp}
v_h=w_h+\sum_{T\in\T} v_h^T,
\end{equation}
where $w_h\in\Lambda^0_h$ is supported in $\Sigma_h$ and $v^T_h\in\Lambda^0_h$ is supported in $T$.
Indeed, we just take $w_h$ to coincide with $v$ at the vertices of the triangulation contained in
$K$ and to vanish at the other vertices, while $v^T_h=v$ at the vertices
in the interior of $T$ and vanishes at the other vertices. 
Because the mesh family is shape regular and quasiuniform, there exist positive constants
$C,c$ such that
$$
c\|v\|^2 \le h \sum_{x\in\Delta_0(\T_h)} |v(x)|^2 \le C\|v\|^2, \quad v\in\Lambda^0_h,
$$
from which we obtain the stability bound
\begin{equation}\label{e:l2-stab}
\|w_h\| + \sum_{T\in\T}\|v^T_h\| \le C\|v_h\|.
\end{equation}

Using the decomposition \eqref{e:decomp} of $v_h$ we get
\begin{equation}
\begin{split}
|\langle \pi_hu-u, \exd v_h \rangle| 
&\leq 
|\langle \pi_hu-u, \exd w_h \rangle| + \sum_{T\in\T}|\langle \pi_h u-u, \exd v_h^T \rangle|\\
&\leq 
C h\|u\|_{H^\ell\Lambda^1(\Sigma_h)}\|\exd w_h\|
+ Ch^2\sum_{T\in\T}\|u\|_{H^\ell \Lambda^1(M)}\|\exd v_h^T\|\\
&\leq 
C \|u\|_{H^\ell\Lambda^1(\Sigma_h)}\|w_h\|_{L^2(\Sigma_h)}
+ Ch\sum_{T\in\T}\|u\|_{H^\ell \Lambda^1(M)}\|v_h^T\|\\
&\leq 
C \left( \|u\|_{H^\ell \Lambda^1(\Sigma_h)} + h\|u\|_{H^\ell \Lambda^1(M)} \right) \|v_h\|,
\end{split}
\end{equation}
where we have used the Cauchy--Schwarz inequality,
the projection error estimate \eqref{proj-estimate},
the second order estimate \eqref{e:h2-bd} (which holds on the uniform meshes
on each $T$),
the inverse estimate of \eqref{e:inv}, and the $L^2$-stability bound \eqref{e:l2-stab}.
Since the volume of $\Sigma_h$ goes to $0$ as $h\to0$, so does
$\|u\|_{H^\ell \Lambda^1(\Sigma_h)}$.  Thus $A_h(u)$ vanishes with $h$,
and the desired result is a consequence of Theorem~\ref{t:equiv}.
\end{proof}

\begin{remark}
The preceding proof shows that as long as the triangulation is mostly uniform,
in the sense that the volume of the defective region goes to $0$ as $h\to0$,
we obtain consistency.
One can also extract information on the convergence rate.
For instance, using the fact  that $\Sigma_h$ is $O(h)$,
we obtain $\|u\|_{H^\ell \Lambda^1(\Sigma_h)}\leq C\sqrt{h}\|u\|_{C^\ell \Lambda^1}$ for $u\in C^\ell \Lambda^1(M)$.
\end{remark}

\section{Computational experiments for $1$-forms}\label{s:experiments}
In this section, we present numerical computations confirming the consistency of $d^*_h$
for $1$-forms on uniform and
piecewise uniform meshes in $2$ and $3$ dimensions, and other computations confirming its
inconsistency on more general meshes.
The four tables in this section display the results of computations with various
mesh sequences.  In each case we show the maximal simplex diameter $h$,
the number of simplices in the mesh, the consistency error $\|d^*_h\pi_h f - d^*f\|$,
and the apparent order inferred from the ratio of consecutive errors.
All computations were performed using the FEniCS finite element software library
\cite{LoggMardalEtAl2012a}.

The first two tables concern the problem on the square
described in Section~\ref{s:counterexample}, i.e., the approximation of $d^*u$
where $u=(1-x^2)dx$.
Table~\ref{tb:unstruct-2d} shows the results when
the piecewise uniform mesh sequence shown in Figure~\ref{f:puniform} is
used for the discretization.  Notice that the consistency error clearly tends
to zero as $O(h)$.
\begin{table}[ht]
\caption{When computed using the $2$-dimensional piecewise uniform mesh sequence of Figure~\ref{f:puniform},
the consistency error tends to $0$.}
\label{tb:unstruct-2d}
\begin{tabular}{rrrrr}
\multicolumn{1}{c}{$n$} & \multicolumn{1}{c}{$h$} & \multicolumn{1}{c}{triangles} & \multicolumn{1}{c}{error} & \multicolumn{1}{c}{order} \\
\hline
  1 &  5.00e$-$01 &     20  & 6.25e$-$01 & \\
  2 &  2.50e$-$01 &    80   & 3.08e$-$01   & 1.02 \\
  3 &  1.25e$-$01 &    320 &  1.56e$-$01  &  0.98 \\
  4 &  6.25e$-$02 &   1,280  & 7.85e$-$02 &   0.99 \\
  5 &  3.12e$-$02 &   5,120  & 3.94e$-$02  &  1.00 \\
  6 &  1.56e$-$02 &  20,480 &  1.97e$-$02 &   1.00 \\
\hline
\end{tabular}
\end{table}

By contrast, Table~\ref{tb:standard-2d} shows the counterexample described analytically
in Section~\ref{s:counterexample},
using the mesh sequence of Figure~\ref{f:crisscross}, obtained by standard subdivision.  In this case, the consistency
error does not converge to zero, as is clear from the computations.
\begin{table}[ht]
\caption{With the mesh sequence of Figure~\ref{f:crisscross}, the consistency
error does not tend to $0$.}
\label{tb:standard-2d}
\begin{tabular}{rrrrr}
 \multicolumn{1}{c}{$n$} & \multicolumn{1}{c}{$h$} & \multicolumn{1}{c}{triangles} & \multicolumn{1}{c}{error} & \multicolumn{1}{c}{order} \\
\hline
  1 &  5.00e$-$01 &     16  & 1.15 & \\
  2 &  2.50e$-$01 &     64  & 1.50  & $-$0.38 \\
  3 &  1.25e$-$01 &    256 & 1.60  & $-$0.09 \\
  4 &  6.25e$-$02 &   1,024 & 1.62 &  $-$0.02 \\
  5 &  3.12e$-$02 &   4,096 & 1.63 &  $-$0.01 \\
  6 &  1.56e$-$02 &  16,384 &  1.63 &  $-$0.00 \\
\hline
\end{tabular}
\end{table}

Similar results hold in $3$ dimensions. We computed
the error in $d^*_hu$ on the cube $(-1,1)^3$, where again
$u$ is given by $(1-x^2)dx$.  We calculated with two mesh sequences, both
starting from a partition of the cube into
six congruent tetrahedra, all sharing a common edge along the diagonal
from $(-1,-1,-1)$ to $(1,1,1)$.
We constructed the first mesh sequence by regular subdivision, yielding the meshes
shown in Figure~\ref{f:meshes-3d-reg}.  These are uniform meshes,
and the numerical results given in  Table~\ref{tb:regular-3d} clearly
demonstrate consistency.
For the second mesh sequence we applied
standard subdivision, obtaining the sequence of structured but non-uniform triangulations
shown in Figure~\ref{f:meshes-3d-std}.  In this case $d^*_h$ is
inconsistent.  See Table~\ref{tb:standard-3d}.

\begin{figure}[htb]
\centerline{%
 \begin{tabular}{cc}
  \includegraphics[width=1.2in]{figures/unit_cube_6_tetrahedra_regular_subdivision/mesh1.png} &
  \includegraphics[width=1.2in]{figures/unit_cube_6_tetrahedra_regular_subdivision/mesh2.png} \\
  \includegraphics[width=1.2in]{figures/unit_cube_6_tetrahedra_regular_subdivision/mesh3.png} &
  \includegraphics[width=1.2in]{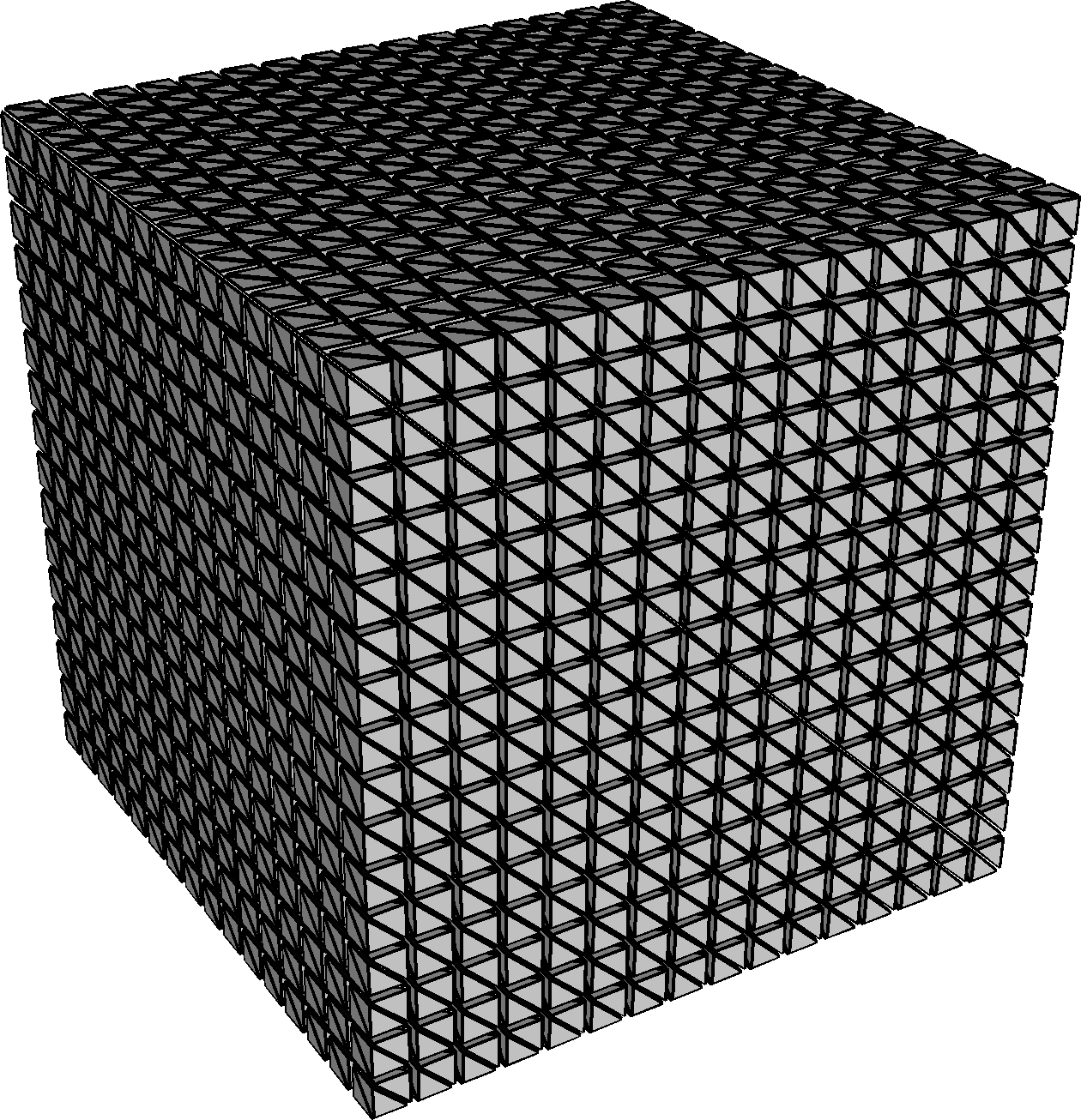}
 \end{tabular}}
\caption{Uniform mesh sequence in 3D, obtained by regular subdivision.}
\label{f:meshes-3d-reg}
\end{figure}

\begin{figure}[htb]
\centerline{%
 \begin{tabular}{cc}
  \includegraphics[width=1.2in]{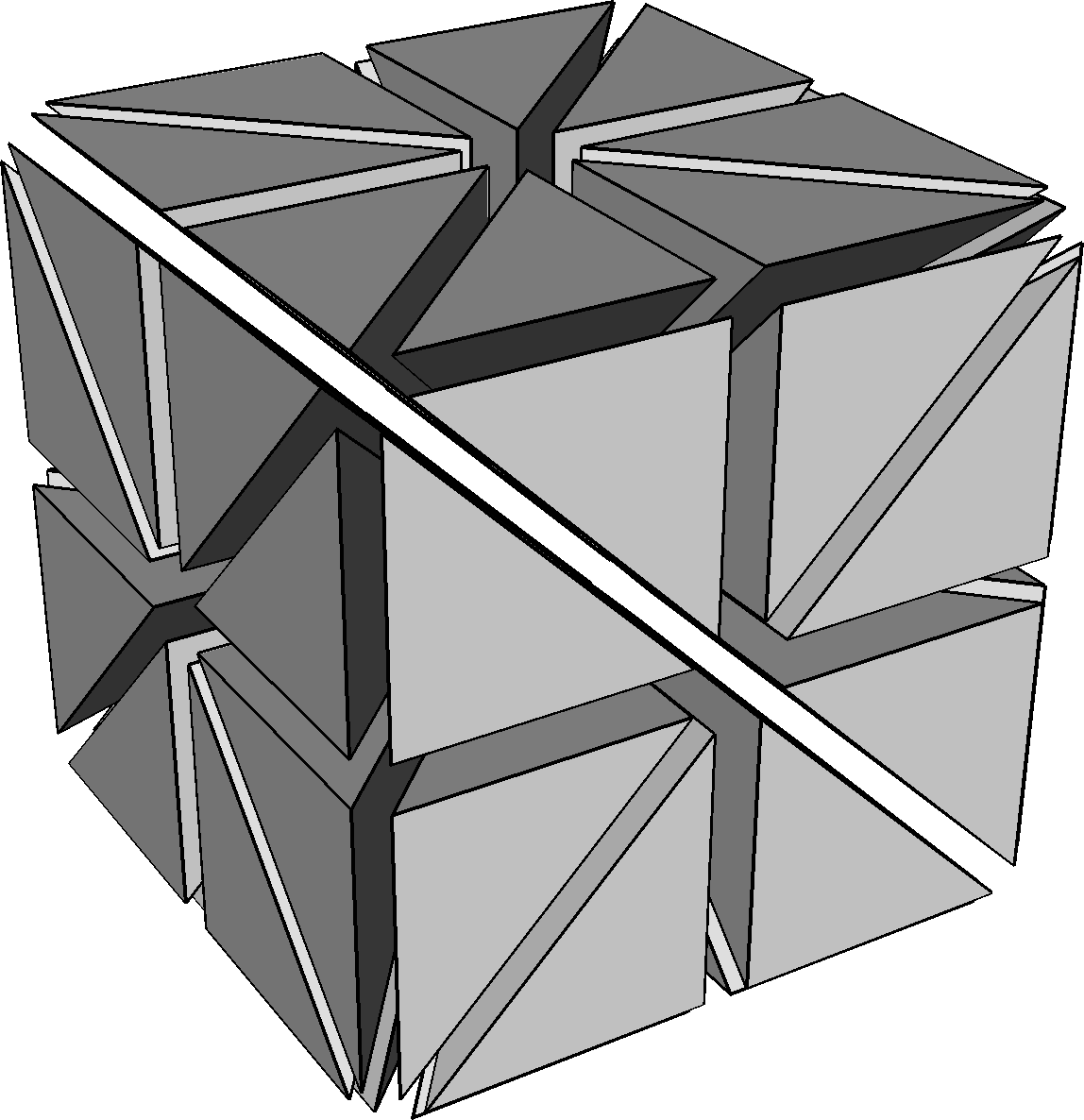} &
  \includegraphics[width=1.2in]{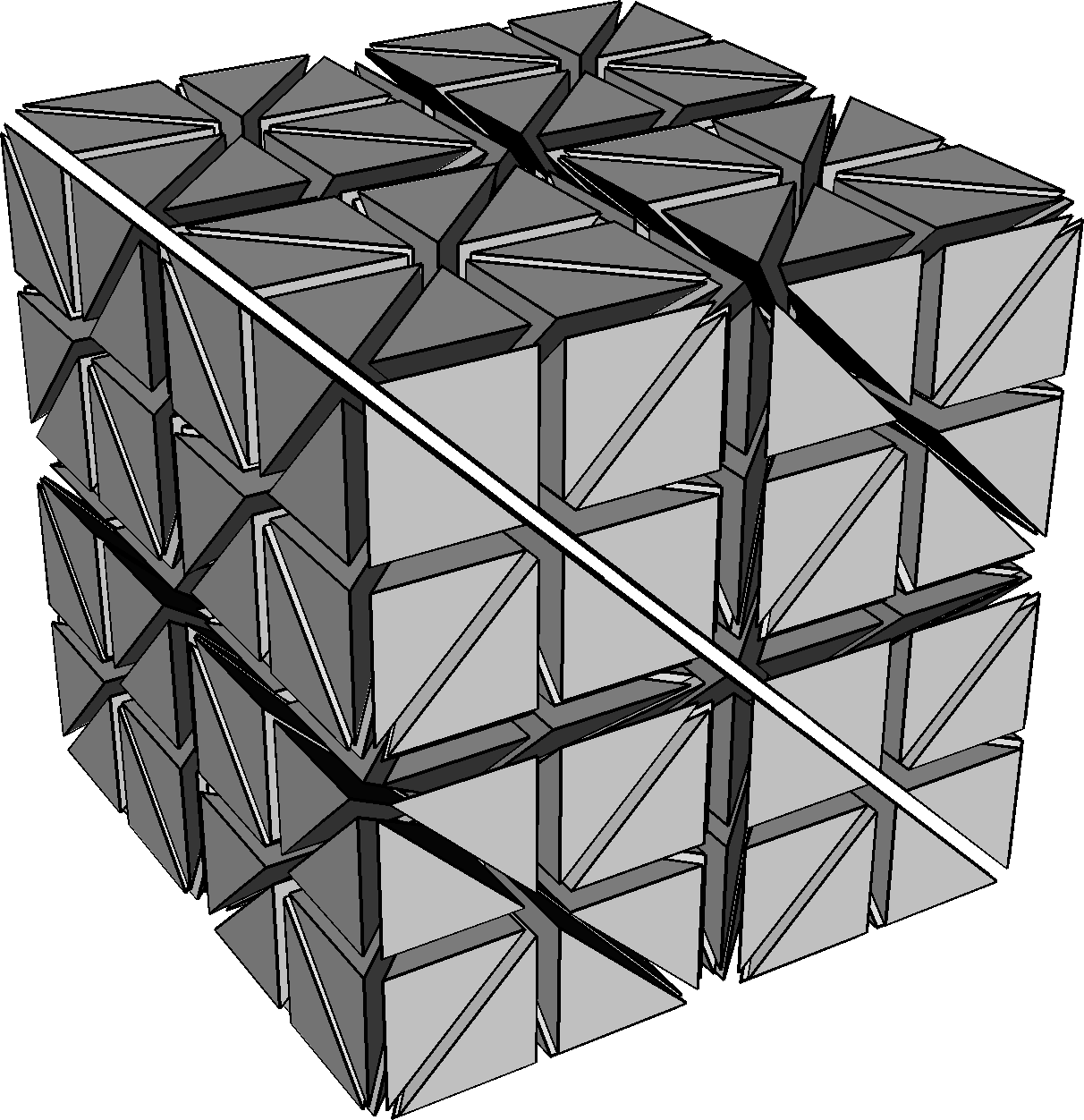} \\
  \includegraphics[width=1.2in]{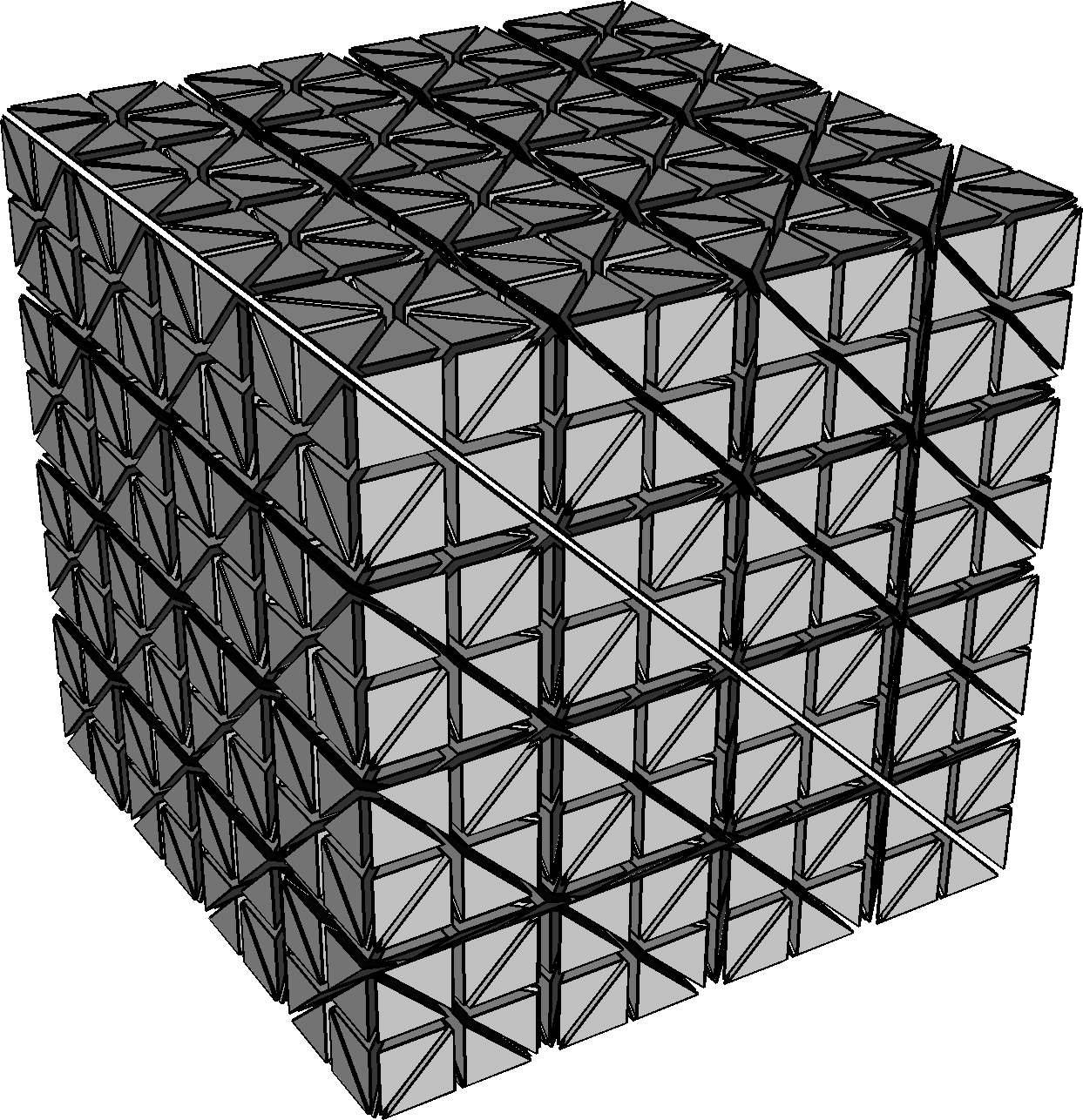} &
  \includegraphics[width=1.2in]{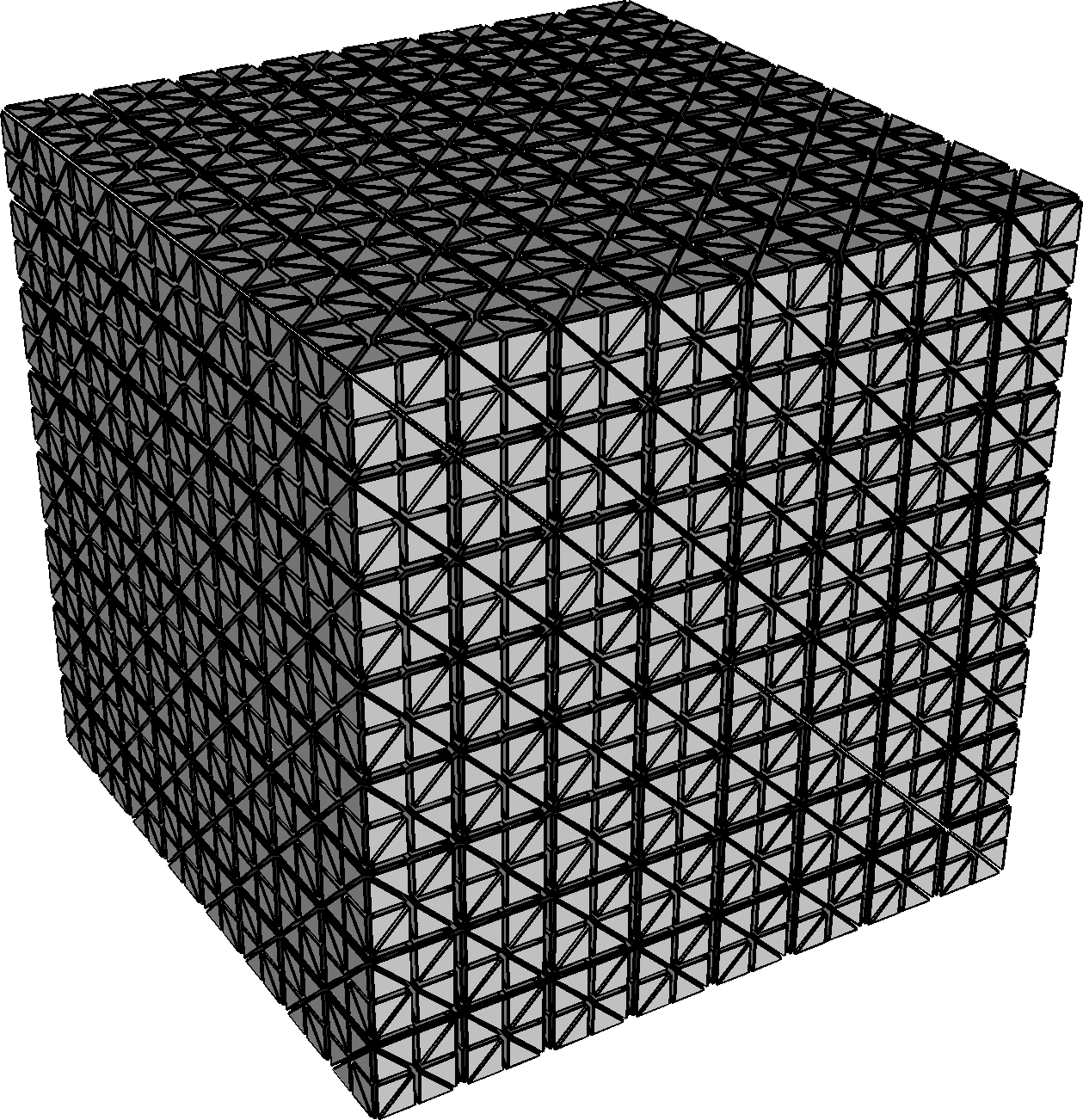}
 \end{tabular}}
\caption{As in 2D, the mesh sequence in 3D obtained by standard subdivision is not uniform.}
\label{f:meshes-3d-std}
\end{figure}

\begin{table}[ht]
\caption{The consistency error for $d^*_h$ on $1$-forms in $3$D tends
to zero when using
the uniform mesh sequence of Figure~\ref{f:meshes-3d-reg}.}
\label{tb:regular-3d}
\begin{tabular}{rrrrr}
 \multicolumn{1}{c}{$n$} & \multicolumn{1}{c}{$h$} & \multicolumn{1}{c}{tetrahedra} & \multicolumn{1}{c}{error} & \multicolumn{1}{c}{order} \\
\hline
  1 &  1.00e+00 &     48 &  1.69e+00 &\\
  2  & 5.00e$-$01  &   384 & 9.70e$-$01   & 0.80\\
  3  & 2.50e$-$01  &  3,072 &  5.13e$-$01  &  0.92\\
  4  & 1.25e$-$01  & 24,576 &  2.63e$-$01  &  0.96\\
  5  & 6.25e$-$02  & 196,608 &  1.33e$-$01 &   0.98\\
  6  & 3.12e$-$02  & 1,572,864 &  6.69e$-$02 &   0.99\\
\hline
\end{tabular}
\end{table}

\begin{table}[ht]
\caption{The consistency error for $d^*_h$ on $1$-forms in $3$D, using
the non-uniform mesh sequence of Figure \ref{f:meshes-3d-std},
does not tend to zero.}
\label{tb:standard-3d}
\begin{tabular}{rrrrr}
 \multicolumn{1}{c}{$n$} & \multicolumn{1}{c}{$h$} & \multicolumn{1}{c}{tetrahedra} & \multicolumn{1}{c}{error} & \multicolumn{1}{c}{order} \\
\hline
  0 &  1.00e+00 &     48  & 1.81e+00 &\\
  1 &  5.00e$-$01  &   384  & 2.71e+00  & $-$0.58\\
  2 & 2.50e$-$01   & 3,072  & 3.02e+00  & $-$0.16\\
  3 & 1.25e$-$01   & 24,576 &  3.11e+00 &   $-$0.04\\
  4 & 6.25e$-$02   & 196,608 &   3.13e+00 &  $-$0.01\\
\hline
\end{tabular}
\end{table}

\section{Inconsistency for $2$-forms in $3$ dimensions}\label{s:2-forms}
We have seen that for $1$-forms, $d^*_h$ is consistent if computed using
piecewise uniform mesh sequences, but not with general mesh
sequences.  It is also easy to see that consistency holds for $n$-forms in $n$-dimensions
for \emph{any} mesh sequence.  This is because
the canonical projection $\pi_h$ onto the Whitney $n$-forms (which are just
the piecewise constant forms) is the $L^2$ orthogonal projection.  Now if $v_h$ is a Whitney $(n-1)$-form,
then $dv_h$ is a Whitney $n$-form, so the inner product
$\<u-\pi_h u,dv_h\>=0$.  Thus $A_h(u)$, defined in \eqref{defA},
vanishes identically, and so $d^*_h$
is consistent by Theorem~\ref{t:equiv}.  Having understood the situation for $1$-forms and $n$-forms,
this leaves open the question of whether consistency holds for $k$-forms
with $k$ strictly between $1$ and $n$.  In this section we study $2$-forms in $3$ dimensions
and give numerical results indicating that $d^*_h$ is not consistent, even for uniform meshes.

Let $u=(1-x^2)(1-y^2)dx\wedge dy$, a $2$-form on the cube $M=(-1,1)^3$.
The corresponding vector field is $(0,0,(1-x^2)(1-y^2))$ which has
vanishing tangential components on $\partial M$.  Therefore $u$ belongs
to the domain of $d^*$ and $d^*u$ is the $1$-form corresponding to $\curl u$,
i.e., $d^*u = -2(1-x^2)ydx+2x(1-y^2)dy$.
Table~\ref{tb:2-forms-3d} shows the consistency error $\|d^*_h\pi_h u - d^*u\|_{L^2\Lambda^1}$
computed using the sequences of uniform meshes displayed in Figure \ref{f:meshes-3d-reg}.
This mesh sequence yields a consistent approximation of $d^*h$ for 1-forms, but the experiments
clearly indicate that this is not so for 2-forms.

\begin{table}[ht]
\caption{The consistency error does not tend to zero for $2$-forms, even on
a uniform mesh sequence.}
\label{tb:2-forms-3d}
\begin{tabular}{rrrrr}
 \multicolumn{1}{c}{$n$} & \multicolumn{1}{c}{$h$} & \multicolumn{1}{c}{triangles} & \multicolumn{1}{c}{error} & \multicolumn{1}{c}{order} \\
\hline
 1 &  1.00e+00 &     48  & 1.59e+00 & \\
 2 &  5.00e$-$01  &  384   & 1.18e+00 &   0.43 \\
 3 &  2.50e$-$01  &  3072 &  1.00e+00 &   0.24 \\
 4 &  1.25e$-$01  & 24576 &  9.47e$-$01 &   0.08 \\
 5 &  6.25e$-$02 & 196608 &  3.37e+00 &   $-$1.83 \\
\hline
\end{tabular}
\end{table}

\bibliographystyle{amsplain}
\bibliography{dstar}

\end{document}